\documentclass[a4paper]{article}
\usepackage{amsmath,amssymb,amsthm}
\usepackage{graphicx,epsfig}
\usepackage{fancyhdr}

\usepackage{color}
\usepackage{amsfonts}
\usepackage{latexsym}
\usepackage{bbm}

\makeatletter\@addtoreset{equation}{section}\makeatother
\newtheorem{remark}{Remark}
\newtheorem{lemma}{Lemma}
\newtheorem{theorem}{Theorem}

\definecolor{lightgreen}{rgb}{0.22,0.50,0.25}
\definecolor{lightblue}{rgb}{0.22,0.45,0.70}
\usepackage[colorlinks=true,breaklinks=true,linkcolor=lightgreen,citecolor=lightblue]{hyperref}

\usepackage{psfrag}\usepackage{cleveref}
\newcounter{constantsnumber}

\newcommand{\CJ}{\mathcal{J}}
\newcommand{\C}{\mathcal{C}}

\newcommand{\bM}{\mbox{\boldmath{$M$}}}

\newcommand{\bA}{\mbox{\boldmath{$A$}}}

\newcommand{\bP}{\mbox{\boldmath{$P$}}}

\newcommand{\bd}{\boldsymbol{d}}

\newcommand{\be}{\boldsymbol{e}}

\newcommand{\fb}{\boldsymbol{f}}

\newcommand{\bH}{\mathbf{H}}

\newcommand{\bn}{\mbox{\boldmath{$n$}}}

\newcommand{\bW}{\mathbf{W}}

\newcommand{\bu}{\boldsymbol{u}}

\newcommand{\bS}{\mbox{\boldmath{$S$}}}
\newcommand{\bt}{\mbox{\boldmath{$t$}}}

\newcommand{\bU}{\mbox{\boldmath{$U$}}}

\newcommand{\bv}{\boldsymbol{v}}
\newcommand{\bV}{\mbox{\boldmath{$V$}}}

\newcommand{\bL}{\mathbf{L}}

\newcommand{\bw}{\boldsymbol{w}}
\newcommand{\bx}{\mbox{\boldmath{$x$}}}
\newcommand{\bX}{\mathbf{X}}
\newcommand{\bY}{\mathbf{Y}}

\newcommand{\by}{\mbox{\boldmath{$y$}}}

\newcommand{\bz}{\mbox{\boldmath{$z$}}}
\newcommand{\bZ}{\mathbf{Z}}

\newcommand{\bepsilon}{\mbox{\boldmath{$\varepsilon$}}}

\newcommand{\bsigma}{\boldsymbol{\sigma}}

\newcommand{\btau}{\boldsymbol{\tau}}

\newcommand{\bxi}{\mbox{\boldmath{$\xi$}}}

\newcommand{\sym}{\rm sym}

\newcommand{\bbL}{\mathbb{L}}
\newcommand{\bbH}{\mathbb{H}}
\newcommand{\bbLsym}{\mathbb{L}_{\mathrm{sym}}^2(\Omega)}
\newcommand{\bbLskew}{\mathbb{L}_{\mathrm{skew}}^2(\Omega)}
\newcommand{\bbHsym}{\mathbb{H}_{\mathrm{sym}}(\bdiv;\Omega)}
\newcommand{\bzero}{\boldsymbol{0}}
\newcommand{\BBR}{\mbox{$\mathbb{R}$}}

\newfont{\twelvemsb}{msbm10 at 11.6pt}

\renewcommand{\div}{\mathop{\rm div}\nolimits}
\newcommand{\bdiv}{\mathop{\mathbf{div}\,}\,}
\renewcommand{\max}{\mathop{\rm max}\limits}

\newcommand{\tr}{\mathop{\rm tr}}

\definecolor{mygrey}{rgb}{0.65,0.65,0.6375}

\newcommand{\cgrey}[1]{{\leavevmode\color{mygrey}{#1}}}

\newcommand\bbR{\mathbb{R}}
\def\bdiv{\mathbf{div}}
\def\qan{{\quad\hbox{and}\quad}}
\def\qin{{\quad\hbox{in}\quad}}
\newcommand{\ds}{\displaystyle}
\def\qon{{\quad\hbox{on}\quad}}
\newcommand\bbI{\mathbb{I}}
\newcommand{\bbeta}{{\boldsymbol\eta}}
\def\rskew{\mathrm{skew}}
\newcommand{\bgamma}{{\boldsymbol{\gamma}}}
\newcommand{\bchi}{{\boldsymbol{\chi}}}

\newcommand{\rH}{\mathrm{H}}
\newcommand{\bzeta}{{\boldsymbol\zeta}}

\def\rt{\mathrm{t}}

\newcommand{\cJ}{\mathcal{J}}
\newcommand{\ubu}{\underline{\bu}}
\newcommand{\ubv}{\underline{\bv}}
\newcommand{\ubd}{\underline{\bd}}
\newcommand{\ube}{\underline{\be}}
\newcommand{\ubw}{\underline{\bw}}
\def\wh{\widehat}
\def\dist{\mathrm{dist}\,}

\numberwithin{equation}{section}
\numberwithin{remark}{section}
\numberwithin{theorem}{section}
\numberwithin{lemma}{section}
\numberwithin{figure}{section}
\numberwithin{table}{section}

\allowdisplaybreaks

\usepackage{colordvi}

\title{New twofold saddle-point formulations for Biot poroelasticity with porosity-dependent permeability\thanks{\textbf{Funding:} This work has been partially supported by the Monash Mathematics Research Fund S05802-3951284; by the Australian Research Council through the \textsc{Future Fellowship} grant FT220100496 and \textsc{Discovery Project} grant DP22010316; by the Ministry of Science and Higher Education of the Russian Federation within the framework of state support for the creation and development of World-Class Research Centers \textsc{Digital biodesign and personalized healthcare}  No. 075-15-2022-304; and by the National Research and Development Agency (ANID) of the Ministry of Science, Technology, Knowledge and Innovation of Chile through the postdoctoral program \textsc{Becas Chile} grant 74220026.}}

\author{Bishnu P.~Lamichhane\thanks{School of Mathematical \& Physical Sciences,
University of Newcastle,
University Drive,
Callaghan, NSW 2308, Australia;  {\tt Bishnu.Lamichhane@newcastle.edu.au}} 
\quad 
Ricardo Ruiz-Baier\thanks{School of Mathematics, Monash University, 9 Rainforest Walk, Melbourne, VIC 3800, Australia; and
                  World-Class Research Center ``Digital biodesign and personalised healthcare", Sechenov First Moscow State Medical University, Moscow, Russia; and
                  Universidad Adventista de Chile, Casilla 7-D, Chill\'an, Chile; {\tt Ricardo.RuizBaier@monash.edu}} \quad 
Segundo Villa-Fuentes\thanks{School of Mathematics, Monash University, 9 Rainforest Walk, Melbourne, VIC 3800, Australia; {\tt Segundo.VillaFuentes@monash.edu}}}

\begin{document}
\maketitle

\begin{abstract}
We propose four-field and five-field Hu--Washizu-type mixed formulations for nonlinear poroelasticity -- a coupled fluid diffusion and solid deformation process -- considering that the permeability depends on a linear combination between fluid pressure and dilation. As the determination of the physical strains is necessary, the first formulation is  written in terms of the primal unknowns of solid displacement and pore fluid pressure as well as the poroelastic stress and the infinitesimal strain, and it considers strongly symmetric Cauchy stresses. The second formulation imposes stress symmetry in a weak sense and it requires the additional unknown of solid rotation tensor. We study the unique solvability of the problem using the Banach fixed-point theory, properties of twofold saddle-point problems, and the Banach--Ne\v{c}as--Babu\v{s}ka theory. We propose monolithic Galerkin discretisations based on  conforming Arnold--Winther for poroelastic stress and displacement, and either PEERS or Arnold--Falk--Winther finite element families for the  stress-displacement-rotation field variables. The wellposedness of the discrete problem is established as well, and we show a priori error estimates in the natural norms. Some numerical examples are provided to confirm the rates of convergence predicted by the theory, and we also illustrate the use of the formulation in some typical tests in Biot poroelasticity. 
\end{abstract}



\section{Introduction}
\subsection{Scope}
The coupling of interstitial fluid flow and solid mechanics in a porous medium has an important role in a number of socially relevant applications \cite{borregales2021iterative}. In particular, nonlinear poroelasticity equations arise, for example, in models of geomechanics and in the study of deformable soft tissues (such as filtration of aqueous humor through cartilage-like structures in the eye and with application in glaucoma formation). In the present work we focus on the case of fully saturated deformable porous media (the solid and fluid constituents of the mixture occupy a complementary fraction of volume in the macroscopic body) and in instances where the permeability  coefficient depends on the porosity (which is in turn related to the total amount of fluid in the poroelastic mixture) \cite{coussy04}. In the classical form for this class of problems the momentum and mass balance equations for a solid-fluid mixture are written in terms of the solid displacement of the porous matrix and the averaged interstitial pressure. Examples of analysis of existence and uniqueness of solution can be found in \cite{bociu2016analysis,bociu2021nonlinear,cao2013analysis,gaspar2016numerical,showalter2001partially,tavakoli2013existence,van2023mathematical}. These works expand the theory available for linear Biot consolidation problems using, for example, constructive Galerkin approximations together with Brouwer's fixed-point arguments with compactness and passage to the limit, the theory of monotone operators in Banach spaces and semigroups, abstract results on doubly nonlinear evolution equations, and the modification of the arguments to the case of pseudo-monotone nonlinear couplings using Br\'ezis' theory. One of the goals of this paper is to extend the previous analysis to the case of mixed formulations for the solid phase. Rewriting the governing equations in mixed form using the Hellinger--Reissner principle and writing the total poroelastic stress as a new unknown, is an approach employed already in the analysis of a number of mixed models for linear poroelasticity  \cite{ambartsumyan2020coupled,baerland17,elyes18,lee16,yi2014convergence} and in poroelasticity/free-fluid couplings \cite{ambartsumyan2019nonlinear,caucao2022multipoint,li20}. These formulations may incorporate also the tensor of rotations to impose in a weak manner the symmetry of the poroelastic stress tensor. There, in deriving the weak forms, one tests the constitutive equation for stress against a test function for stress.  In contrast, in the present treatment we also require the tensor of infinitesimal strains as an unknown since it is an important field acting in the coupling with the fluid phase mechanics via the nonlinear permeability (regarded as a function of the interstitial fluid pressure and the trace of the strain tensor).  In the context of  elasticity problems, the popular Hu--Washizu formulation \cite{Hu, Washizu} has  displacement, stress and strain tensors as three unknowns, and we note that the Hellinger--Reissner formulation mentioned above is a special case of the Hu--Washizu formulation (it is obtained after applying the Fenchel--Legendre transformation eliminating the strain from the latter formulation \cite{Bra01}).
 
Apart from the application of  Hu--Washizu formulations in many works for linear elasticity (see, e.g., \cite{DLR06,djoko06,LRW06,lamperti23,wagner20} and the references therein), the solvability analysis of the continuous and discrete twofold saddle-point mixed problems (including also error estimates) has been carried out in \cite{gatica2007dual,gatica2013priori} for Hencky-strain nonlinear elasticity, as well as for more recent models for stress-assisted diffusion coupled with poroelasticity  \cite{gomez23}. There, one tests the constitutive equation for stress against the test function associated with the space of infinitesimal strains. In that setting, a key  ingredient in the analysis is the assumption that the nonlinearity in the weak forms induces a Lipschitz continuous and strongly monotone operator (this last condition being required in a suitable kernel).  

In the present scenario the analysis requires to define a nonlinear operator $\bA:\bX\to \bX'$ where $\bX$ consists of square integrable and symmetric tensors and scalar functions in $\rH^1$. The nonlinearity is inherited from the nonlinear dependence of permeability on fluid pressure and on skeleton strains, and for some constitutive forms, it does not necessarily imply that $\bA$ is monotone. In this work we present two formulations, in the first we impose the symmetry of the stress tensor in a strong way, while in the second we impose the symmetry in a weak sense, using the rotation tensor as a further unknown. Then, similarly to \cite{cgs2021} (see also \cite{gatica03}) we have the nonlinear term inside the saddle-point structure, unlike \cite{CGO2021,cov2020} where the nonlinear term is associated with a perturbation of the saddle-point problem. Therefore we proceed by a fixed-point argument and consider a linear twofold saddle-point formulation that suggests the structure of a fixed-point operator. This map is shown to be well-defined (for this we use an appropriate adaptation of the theory from, e.g., \cite{howell2011inf}), to map a conveniently chosen ball into itself, and to be Lipschitz continuous. Then, by establishing a contracting property the unique solvability will be a consequence of Banach fixed-point theorem. Such an analysis hinges on data smallness assumptions, which involve boundary data, source terms, and permeability bounds. 

For the associated Galerkin schemes we employ, on the one hand, for the strong symmetry formulation, Arnold--Winther finite elements of degree $k \geq 1$ \cite{arnold-winther-0} to approximate the strain tensor, poroelastic stress tensor and displacement, and continuous piecewise polynomials of degree $k+1$ for the pore pressure; and we note that for this finite element family, one can also employ piecewise polynomials of degree $k+2$ for the symmetric strain tensor to maintain inf-sup stability.  On the other hand, for the weak symmetry formulation, we use the classical PEERS elements \cite{arnold84} to approximate the strain tensor, poroelastic stress tensor, displacement and the rotation tensor, and continuous piecewise polynomials of degree $k+1$ for the interstitial pressure (we also use a family based on Arnold--Falk--Winther elements \cite{arnold-2007}). Next we apply the same arguments utilised for the continuous problem to prove unique solvability. In addition, using standard tools and techniques for the error decomposition, and approximation properties of mentioned finite element spaces, we obtain the corresponding C\'ea estimate and rates of convergence.

\subsection{Outline} 
The content of this paper has been laid out as follows. In the remainder of this section we include notation conventions and preliminary results that are used throughout the manuscript. Section~\ref{sec:model} provides details of the model, describing the components of the balance equations and stating boundary conditions. Two weak forms and their properties are collected in Section~\ref{sec:weak}. Section~\ref{sec:wellp} is devoted to the analysis of solvability of the weak forms, using  arguments from the twofold saddle-point variant of the {Babu\v ska--Brezzi} theory from \cite{gatica03}. Next we address in Section \ref{sec:FE} the solvability and stability analysis of the discrete problem, where similar arguments are employed, following \cite{gatica2013priori}. A priori error estimates are derived in Section~\ref{sec:error},  and in Section~\ref{sec:numer} we collect   computational results, consisting in verification of  convergence and simulation of different cases on simple geometries. 

\subsection{Notation and preliminaries}

Let  $L^2(\Omega)$ be the set of all square-integrable functions in 
$\Omega \subset \BBR^d$ where $d \in \{2,3\}$ is the spatial dimension, and denote by $\bL^2(\Omega)=L^2(\Omega)^d$ its vector-valued counterpart and by $\bbL^2(\Omega)=L^2(\Omega)^{d\times d}$ its tensor-valued counterpart. We also write  
\[\bbLsym:=\{\btau \in \bbL^2(\Omega):\,\btau = \btau^{\tt t} \}, \quad 
\bbLskew:=\{\btau \in \bbL^2(\Omega):\,\btau = -\btau^{\tt t} \},
\]
to represent the symmetric and skew-symmetric tensors in $\Omega$ with each  component being square-integrable. Standard notation will be employed for Sobolev spaces $\rH^m(\Omega)$ with $m\geq 0$ (and we note that $\rH^0(\Omega)=L^2(\Omega)$). Their norms and seminorms are denoted as $\|\cdot\|_{m,\Omega}$ and $|\cdot|_{m,\Omega}$, respectively (as well as for their vector and tensor-valued counterparts $\bH^m(\Omega)$, $\bbH^m(\Omega)$) see, e.g., \cite{BS94}. 

As usual $\mathbb I$ stands for the identity tensor in $\bbR^{d\times d}$, 
and $|\cdot|$ denotes the Euclidean norm  in $\bbR^d$. Also, for any vector fields $\bv=(v_i)_{i=1,d}$  we set the gradient and divergence operators as
\[
\nabla \bv \,:=\, \left(\frac{\partial v_i}{\partial x_j}\right)_{i,j=1,d}\, \qan
\div\,\bv \,:=\, \sum_{j=1}^d \frac{\partial v_j}{\partial x_j}\,.
\]
In addition, for any tensor fields $\btau=(\tau_{ij})_{i,j=1,d}$
and $\bzeta = (\zeta_{ij})_{i,j=1,d}$, we let $\bdiv\,\btau$ be the divergence operator $\div$ acting along the rows of $\btau$, and define the transpose, the trace and the tensor inner product, respectively, as
\begin{equation*}
\btau^\rt := (\tau_{ji})_{i,j=1,d},\quad
\tr(\btau) := \sum_{i=1}^d\tau_{ii},\qan 
\btau:\bzeta := \sum_{i,j=1}^n\tau_{ij}\zeta_{ij}.
\end{equation*}
We also recall the Hilbert space
\[
\bH(\div;\Omega)\,:=\,\big\{\bz \in \bL^2(\Omega):\, \div\,\bz \in \mathrm{L}^2(\Omega)\big\},
\]
with norm $\|\bz\|_{\div;\Omega}^{2}:=\|\bz\|_{0,\Omega}^{2}+\|\div\,\bz\|_{0,\Omega}^{2}$,
and introduce the tensor version of $\bH(\div;\Omega)$ given by
\[
\bbH(\bdiv;\Omega)\,:=\,\big\{ \btau \in \bbL^2(\Omega):\quad \bdiv\,\btau \in \bL^2(\Omega)\big\}\,,
\]
whose norm will be denoted by $\|\cdot\|_{\bdiv;\Omega}$.


\subsection{Governing equations}\label{sec:model}
Let us consider a fully-saturated poroelastic medium (consisting of a mechanically isotropic and homogeneous fluid-solid mixture) occupying the open and bounded domain $\Omega$ in $\BBR^d$, with Lipschitz boundary $\Gamma$. The symbol $\bn$ will stand for the unit outward normal vector on the boundary. Let $\fb \in \bL^2(\Omega)$ be a prescribed body force per unit of volume (acting on the fluid-structure mixture) and let $g \in L^2(\Omega)$ be a net volumetric fluid production rate.  

Under the assumption of negligible gravitational effects as well as material deformations being sufficiently small, and varying sufficiently slowly so that inertial effects are considered negligible (see further details in \cite{showalter2000diffusion}), we have that  
 the balance of linear momentum for the solid-fluid mixture is written as 
\begin{equation}
\label{eq:mom} -\bdiv\, \bsigma = \fb \qquad \text{in $\Omega$},
\end{equation}
with $\bsigma$ being the total Cauchy stress tensor of the mixture (conformed by the effective solid stress and effective fluid stress), whose dependence on strain and on fluid pressure  is given by the constitutive assumption (or effective stress principle) 
\begin{equation}\label{eq:constitutive}
  \bsigma = \C \bd -\alpha p \mathbb{I} \qquad \text{in $\Omega$}.
\end{equation}
Here the skeleton displacement vector $\bu$ from the position $\bx \in \Omega$ is an unknown, the tensor $\bd = \bepsilon (\bu) := \frac{1}{2} (\nabla \bu + [\nabla \bu]^{\tt t} )$ is the infinitesimal strain, by $\C $ we denote the fourth-order elasticity tensor, also known as Hooke's tensor (symmetric and positive definite and  characterised by $\C \bd := \lambda(\tr\bd)\mathbb{I} + 2\mu\,\bd$), $\mathbb{I}$ is the identity second-order tensor, $\lambda$ and $\mu$ are the Lam\'e parameters (assumed constant and positive), $0\leq \alpha \leq 1$ is the Biot--Willis parameter, and $p$ denotes the Darcy fluid pressure (positive in compression), which is an unknown in the system. 

We also consider the balance of angular momentum, which in this context states that the total poroelastic stress is a symmetric tensor 
\begin{equation}\label{eq:angular}
    \bsigma = \bsigma^{\tt t}.
\end{equation}

The fluid content (due to both fluid saturation and local volume dilation) is given by 
\[\zeta = c_0 p + \alpha \div \bu,\]
where $c_0$ is the constrained specific storage (or storativity) coefficient. Using Darcy's law to describe the discharge velocity in terms of the fluid pressure gradient, we can write the balance of mass for the total amount of fluid in the mixture as $\partial_t\zeta - \div (\kappa(\phi) \nabla p)  = g$ in $\Omega \times (0,t_{\mathrm{end}})$, where $\kappa$ is the intrinsic permeability (divided by the fluid viscosity) of the laminar flow in the medium, a nonlinear function of porosity $\phi$.  In turn, in the small strains limit the porosity can be approximated by a linear function of the fluid content (see for example \cite[Section 2.1]{van2023mathematical}), and so -- making abuse of notation -- we can simply write $\kappa(\zeta)$.  Furthermore, after a backward Euler semi-discretisation in time with a constant time step and rescaling appropriately, we only consider the type of equations needed to solve at each time step and therefore we will concentrate on the form 
\begin{equation}\label{eq:mass}
c_0 p + \alpha \tr\bd - \div (\kappa(\bd,p) \nabla p)  = g \qquad \text{in $\Omega$}.
\end{equation}
Typical constitutive relations for permeability are, for example, of exponential or Kozeny--Carman type (see, e.g., \cite{ateshian10}) 
\begin{equation}\label{eq:kappa}
    \kappa(\zeta) = \frac{k_0}{\mu_f}\mathbb{I} + \frac{k_1}{\mu_f}\exp(k_2 \zeta)\mathbb{I}, \quad \kappa(\zeta) =   \frac{k_0}{\mu_f}\mathbb{I} +  \frac{k_1\zeta^3}{\mu_f(1-\zeta)^2}\mathbb{I},
\end{equation}
where $\mu_f$ denotes the viscosity of the interstitial fluid and $k_0,k_1,k_2$ are model constants. We note that in the case of incompressible constituents one has $c_0 = 0 $ and $\alpha = 1$, indicating that permeability depends only on the dilation $\tr\bd = \div\bu$ (see, e.g., \cite{bociu2016analysis}). We also note that even in such a scenario (of incompressible phases) the overall mixture is not necessarily incompressible itself. 
More precise assumptions on the behaviour of the permeability are postponed to Section~\ref{sec:wellp}. 

To close the system of equations, we consider non-homogeneous displacement boundary conditions for the momentum balance and non-homogeneous flux boundary conditions on the mass balance equation. For prescribed $\bu_\Gamma \in \bH^{1/2}(\Gamma)$ and $r_\Gamma \in \rH^{-1/2}(\Gamma)$ we set 
\begin{equation}
\label{eq:bc}
  \bu = \bu_\Gamma \quad \text{and}\quad  \kappa(\bd,p)\nabla p\cdot \bn = r_\Gamma \qquad \mbox{on } \Gamma.
\end{equation}


\section{Two weak formulations and preliminary properties}\label{sec:weak}
\subsection{Derivation of weak forms}
We proceed to test equation \eqref{eq:mom} against $\bv\in\bL^2(\Omega)$, to test the constitutive equation for strain $\bd = \bepsilon (\bu)$ against $\btau\in\bbHsym\,:=\,\big\{ \btau \in \bbLsym:\  \bdiv\,\btau \in \bL^2(\Omega) \big\}$, the equations \eqref{eq:constitutive} and \eqref{eq:mass}, by $\be\in \bbLsym$ and $q\in \rH^1(\Omega)$, respectively, integrate by parts and using the boundary conditions \eqref{eq:bc} naturally, we finally arrive at 
\begin{equation}\label{eq:weak-formulation-1}
\begin{array}{rlll}
\ds-\int_{\Omega} \bv \cdot \bdiv\, \bsigma   &=& \ds \int_{\Omega}  \fb \cdot \bv & \ds \forall\,\bv\in\bL^2(\Omega), \\ [2ex]
\ds-\int_{\Omega} \btau:\bd -\int_{\Omega} \bu \cdot \bdiv\, \btau  &=& \ds -\langle \btau\bn,\bu_\Gamma\rangle_\Gamma & \ds \forall\,\btau\in\bbHsym , \\ [2ex]
\ds\int_{\Omega} \C \bd:\be -\alpha \int_{\Omega} p\, \tr\be -\int_{\Omega} \bsigma:\be &=& 0 & \ds \forall\,\be\in\bbLsym, \\ [2ex]
\ds\int_{\Omega} \kappa(\bd,p) \nabla p \cdot \nabla q + c_0 \int_{\Omega} p\,q +  \alpha \int_{\Omega} q\, \tr\bd &=& \ds \int_{\Omega}  g\, q + \langle r_\Gamma,q\rangle_{\Gamma}& \ds \forall\,q\in \rH^1(\Omega),
\end{array}
\end{equation}
where $\langle\cdot,\cdot\rangle_\Gamma$ denotes the duality pairing between $\rH^{-1/2}(\Gamma)$ and its dual $\rH^{1/2}(\Gamma)$ with respect to the inner product in $L^2(\Gamma)$ (and we use the same notation in the vector-valued case). Note also that the balance of angular momentum \eqref{eq:angular} has been enforced as an essential condition in the functional space for poroelastic stress. 

Next, we notice that \eqref{eq:weak-formulation-1} can be regarded as a twofold saddle-point structure. In fact, let us adopt the following notation for  the Hilbert spaces for the  strain-pressure pair, the poroelastic stress, and the displacement:  
\[\bX:=\bbLsym\times \rH^1(\Omega),\qquad \bY:=\bbHsym \qan \bZ:=\bL^2(\Omega),\]
respectively. In addition, we group and order  the trial and test functions as follows:
$$\ubd:=(\bd,p)\in\bX,\quad \bsigma\in \bY, \quad \bu\in \bZ,$$
$$\ube:=(\be,q)\in\bX,\quad \btau\in \bY, \quad \bv\in \bZ,$$
where $\bX$, $\bY$, $\bX\times\bY$ and $\bZ$ are endowed with the norms
$$\|\ube\|^2_\bX:=\|\be\|^2_{0,\Omega} + \|q\|^2_{1,\Omega}, \quad \|\btau\|_{\bY}:=\|\btau\|_{\bdiv,\Omega}, \quad \|(\ube,\btau)\|^2_{\bX\times \bY}:=\|\ube\|^2_\bX + \|\btau\|^2_{\bY}$$
$$\|\bv\|_{\bZ}:=\|\bv\|_{0,\Omega}, \quad \|((\ube,\btau),\bv)\|^2:=\|(\ube,\btau)\|^2_{\bX\times\bY} + \|\bv\|^2_{\bZ}.$$

Introducing the nonlinear and bilinear weak forms $a: \bX\times\bX \to \mathbb{R}$, ${b}_1: \bX \times {\bY} \to \mathbb{R}$ and ${b}_2: (\bX\times\bY) \times {\bZ} \to \mathbb{R}$ defined by  
\begin{equation}\label{eq:def-a-b1-b2}
\begin{array}{ll}
\ds a(\ubd,\ube) & := \ds \int_{\Omega} \C \bd:\be +\int_{\Omega} \kappa(\ubd) \ds\nabla p \cdot \nabla q + c_0 \int_{\Omega} p\,q +  \alpha \int_{\Omega} q\, \tr\bd-\alpha \int_{\Omega} p\, \tr\be  , \\ 
\ds b_1(\ube,\btau)& \ds:= -\int_{\Omega} \btau:\be,\\ 
\ds b_2((\ube,\btau),\bv) &\ds := -\int_{\Omega} \bv \cdot \bdiv\, \btau,
\end{array}
\end{equation}
respectively; and the linear functionals $F\in \bZ'$, $H\in\bY'$, $G\in \bX'$ by 
\begin{equation*}
F(\bv):=\ds \int_{\Omega}  \fb \cdot \bv, \quad 
H(\btau): = -\langle \btau\bn,\bu_\Gamma\rangle_\Gamma, \quad 
G(\ube):=\ds \int_{\Omega}  g\, q  + \langle r_\Gamma,q\rangle_{\Gamma},
\end{equation*}
we can write the weak form \eqref{eq:weak-formulation-1} as follows: Find $((\ubd,\bsigma),\bu)\in (\bX\times\bY)\times\bZ$ such that 
\begin{equation}\label{eq:weak1}
\begin{array}{rlll}
\ds a(\ubd,\ube)  & +\quad b_1(\ube,\bsigma) & \, & = G(\ube), \\ [1ex]
\ds b_1(\ubd,\btau)  & \, & +\quad b_2((\ube,\btau),\bu) & =H(\btau), \\ [1ex]
\,  & \qquad b_2((\ubd,\bsigma),\bv) & \, & = F(\bv),
\end{array}
\end{equation}
for all $((\ube,\btau),\bv)\in (\bX\times\bY)\times\bZ.$

The second weak formulation we treat here results from imposing the symmetry of the poroelastic stress in a weak manner (see, e.g., \cite{gatica14} for the general idea and \cite{lee16} for the application in the context of poroelasticity but leading to a different formulation). In order to do this, it is customary to introduce the rotation tensor 
\begin{equation}
\bgamma = \frac{1}{2} (\nabla \bu - [\nabla \bu]^{\tt t} ) = \nabla \bu -\bd,
\end{equation}
and we can then rewrite the strong form of the coupled PDE system in mixed form as 
\begin{equation}\label{eq:strong-mixed2}
\begin{array}{c}
-\bdiv\, \bsigma = \fb  \qin \Omega,\qquad \bsigma = \bsigma^{\tt t}\qin \Omega,\qquad 
\bgamma = \nabla \bu -\bd \qin \Omega, \\ [2ex]
\bsigma = \C \bd -\alpha p \bbI,  \qin \Omega,\qquad 
c_0 p+\alpha \tr{\bd} - \div (\kappa(\bd,p) \nabla p)   = g, \qin \Omega, \\ [2ex]
\ds \bu = \bu_\Gamma \qon \Gamma,\qquad \kappa(\bd,p) \nabla p\cdot \bn =r_\Gamma \qon \Gamma.
\end{array}
\end{equation}
After testing these equations by $\bv\in\bL^2(\Omega)$, $\bbeta\in\bbLskew$, $\btau\in\bbH (\bdiv;\Omega)$,  
$\be\in\bbL^2(\Omega)$, and $q\in \rH^1(\Omega)$, respectively; we integrate by parts and use \eqref{eq:bc} as natural boundary conditions to 
obtain the system 
\begin{equation*}
\begin{array}{rlll}
\ds-\int_{\Omega} \bv \cdot \bdiv\, \bsigma   &=& \ds \int_{\Omega}  \fb \cdot \bv & \ds \forall\,\bv\in\bL^2(\Omega), \\ [2ex]
\ds-\int_{\Omega} \btau:\bd -\int_{\Omega} \bu \cdot \bdiv\, \btau -\int_{\Omega} \btau:\bgamma  &=& \ds -\langle \btau\bn,\bu_\Gamma\rangle_\Gamma  & \ds \forall\,\btau\in\bbH (\bdiv;\Omega), \\ [2ex]
\ds\int_{\Omega} \C \bd:\be -\alpha \int_{\Omega} p\, \tr\be -\int_{\Omega} \bsigma:\be &=& 0 & \ds \forall\,\be\in\bbL^2(\Omega), \\ [2ex]
\ds\int_{\Omega} \kappa(\bd,p) \nabla p \cdot \nabla q + c_0 \int_{\Omega} p\,q +  \alpha \int_{\Omega} q\, \tr\bd &=& \ds \int_{\Omega}  g\, q + \langle r_\Gamma,q\rangle_{\Gamma}& \ds \forall\,q\in \rH^1(\Omega),\\ [2ex]
\ds-\int_{\Omega} \bsigma:\bbeta  &=& \ds 0 & \ds \forall\,\bbeta\in\bbL^2_{\rskew}(\Omega).
\end{array}
\end{equation*}

Proceeding similarly as in the derivation of \eqref{eq:weak1}, we group spaces, unknowns and test functions as follows:
$$\widetilde{\bX}:=\bbL^2(\Omega)\times \rH^1(\Omega),\quad  \widetilde{\bY}:=\bbH(\bdiv;\Omega), \quad \widetilde{\bZ}:=\bL^2(\Omega)\times \bbL^2_{\rskew}(\Omega),$$
$$\ubd:=(\bd,p)\in\widetilde{\bX},\quad \bsigma\in \widetilde{\bY}, \quad \ubu:=(\bu,\bgamma)\in \widetilde{\bZ},$$
$$\ube:=(\be,q)\in\widetilde{\bX},\quad \btau\in \widetilde{\bY}, \quad \ubv:=(\bv,\bbeta)\in \widetilde{\bZ},$$
where $\widetilde{\bX}$, $\widetilde{\bY}$, $\widetilde{\bX}\times\widetilde{\bY}$ and $\widetilde{\bZ}$ are endowed with the norms
$$\|\ube\|^2_{\widetilde{\bX}}:=\|\be\|^2_{0,\Omega} + \|q\|^2_{1,\Omega}, \quad \|\btau\|_{\widetilde{\bY}}:=\|\btau\|_{\bdiv,\Omega}, \quad \|(\ube,\btau)\|^2_{\widetilde{\bX}\times\widetilde{\bY}}:=\|\ube\|^2_{\widetilde{\bX}} + \|\btau\|^2_{\widetilde{\bY}},$$
$$\|\ubv\|^2_{\widetilde{\bZ}}:=\|\bv\|^2_{0,\Omega} + \|\bbeta\|^2_{0,\Omega}, \quad \|((\ube,\btau),\ubv)\|^2:=\|(\ube,\btau)\|^2_{\tilde{\bX}\times\tilde{\bY}} + \|\ubv\|^2_{\widetilde{\bZ}}.$$

Next, we define the weak forms $\tilde{a}: \widetilde{\bX}\times\widetilde{\bX} \to \mathbb{R}$, $\tilde{b}_1: \widetilde{\bX} \times \widetilde{\bY} \to \mathbb{R}$ and $\tilde{b}_2: (\widetilde{\bX}\times\widetilde{\bY}) \times \widetilde{\bZ} \to \mathbb{R}$ from the expressions 
\begin{equation}\label{eq:def-hat-a-b1-b2}
\begin{array}{ll}
\ds \tilde{a}(\ubd,\ube) & := \ds \int_{\Omega} \C \bd:\be +\int_{\Omega} \kappa(\ubd) \nabla p \cdot \nabla q + c_0 \int_{\Omega} p\,q +  \alpha \int_{\Omega} q\, \tr\bd-\alpha \int_{\Omega} p\, \tr\be  , \\
\tilde{b}_1(\ube,\btau) & \ds:= -\int_{\Omega} \btau:\be ,\\ 
\tilde{b}_2((\ube,\btau),\ubv) &\ds:= -\int_{\Omega} \bv \cdot \bdiv\, \btau -\int_{\Omega} \btau:\bbeta,
\end{array}
\end{equation}
respectively, and the linear functionals $\widetilde{G}\in \widetilde{\bX}'$, $\widetilde{H}\in \widetilde{\bY}'$ and $\widetilde{F}\in \widetilde{\bZ}'$ as 
\begin{equation*}
\widetilde{G}(\ube): = \ds \int_{\Omega}  g\, q  + \langle r_\Gamma,q\rangle_{\Gamma}, \quad 
\widetilde{H}(\btau): = -\langle \btau\bn,\bu_\Gamma\rangle_\Gamma, \quad 
\widetilde{F}(\ubv):=\ds \int_{\Omega}  \fb \cdot \bv,
\end{equation*}
respectively, so that the weak formulation of the nonlinear coupled system \eqref{eq:strong-mixed2} reads: 
Find $((\ubd,\bsigma),\ubu)\in (\widetilde{\bX}\times\widetilde{\bY})\times\widetilde{\bZ}$ such that
\begin{equation}\label{eq:weak2}
\begin{array}{rlll}
\ds \tilde{a}(\ubd,\ube)  & +\quad \tilde{b}_1(\ube,\bsigma) & \, & = \widetilde{G}(\ube), \\ [1ex]
\ds \tilde{b}_1(\ubd,\btau)  & \, & +\quad \tilde{b}_2(\btau,\ubu) & =\widetilde{H}(\btau), \\ [1ex]
\,  & \qquad \tilde{b}_2(\bsigma,\ubv) & \, & = \widetilde{F}(\ubv),
\end{array}
\end{equation}
for all $((\ube,\btau),\ubv)\in (\widetilde{\bX}\times\widetilde{\bY})\times\widetilde{\bZ}.$

\begin{remark}
Note that when $c_0$ approaches zero, our control over the $L^2$-part of the fluid pressure norm diminishes. Consequently, the uniqueness of fluid pressure cannot be guaranteed unless we search for it within a space such as $\rH^1(\Omega)\cap L^2_0(\Omega)$, due to the pure flux boundary conditions imposed on the mass balance equation. Without this consideration, the lack of uniqueness would also extend to stress, as implied by \eqref{eq:constitutive}. In such cases, it becomes necessary to restrict tensors in the spaces $\bY$ and $\widetilde{\bY}$ to those with a zero mean value. A similar scenario arises when $\alpha$ tends to zero: the poroelastic stress loses its unique definition (even though the fluid pressure retains it), requiring the adoption of the zero mean condition within the stress space.
\end{remark}
\subsection{Stability properties and suitable inf-sup conditions}

We start by establishing the boundedness of  the 
bilinear forms  $b_1$, $b_2$, $\tilde{b}_1$ and $\tilde{b}_2$:
\begin{subequations}
\begin{gather}
\label{eq:bounded-b1-b2}
\big|b_1(\ube,\btau)\big|\leq\|\ube\|_{\bX}  \|\btau\|_{\bY},\qquad
\big|b_2((\ube,\btau),\bv)\big|\leq \|(\ube,\btau)\|_{\bX\times\bY}  \|\bv\|_{\bZ},  \\ \label{eq:bounded-hat-b1-b2}
\big|\tilde{b}_1(\ube,\btau)\big|\leq \|\ube\|_{\widetilde{\bX}}  \|\btau\|_{\widetilde{\bY}} ,\qquad \big|\tilde{b}_2((\ube,\btau),\ubv)\big|\leq \|(\ube,\btau)\|_{\widetilde{\bX}\times\widetilde{\bY}} \|\ubv\|_{\widetilde{\bZ}}. 
\end{gather}
\end{subequations}

On the other hand, using H\"older and trace inequalities we can readily observe that the right-hand side functionals are all bounded
\begin{gather}\label{des:bound-F-H-G}
 \big|G(\ube)\big|  \lesssim (\|g\|_{0,\Omega}+ \|r_\Gamma \|_{-1/2,\Gamma})\|\ube\|_{\bX} \quad \forall \ube \in \bX,
 \qquad 
 \big|\widetilde{G}(\ube)\big|  \lesssim (\|g\|_{0,\Omega}+ \|r_\Gamma \|_{-1/2,\Gamma})\|\ube\|_{\widetilde{\bX}} \quad \forall \ube \in \widetilde{\bX},\nonumber \\
 \big|H(\btau)\big|  \lesssim \|\bu_\Gamma\|_{1/2,\Gamma} \|\btau\|_{\bY} \quad \forall \btau \in \bY,\qquad 
 \big|\widetilde{H}(\btau)\big|  \lesssim \|\bu_\Gamma\|_{1/2,\Gamma} \|\btau\|_{\widetilde{\bY}} \qquad \forall \btau \in \widetilde{\bY},\\
 \big|F(\bv)\big| \leq \|\fb\|_{0,\Omega}\, \|\bv\|_{\bZ} \quad \forall \bv \in  \bZ, \qquad 
\big|\widetilde{F}(\ubv)\big| \leq \|\fb\|_{0,\Omega}\, \|\ubv\|_{\widetilde{\bZ}} \quad \forall \ubv \in  \widetilde{\bZ}.\nonumber
\end{gather}

Finally, it is straightforward to see that the kernels of the bilinear forms $b_2$ and $\tilde{b}_2$ are closed subspaces of $\bX\times\bY$ and $\widetilde{\bX}\times\widetilde{\bY}$, respectively. They are denoted  as
\begin{equation}\label{eq:kernel-b2-hat-b2}
\begin{array}{cc}
\bX\times\bY_0 \qan 
\widetilde{\bX}\times\bY_0, 
\end{array}
\end{equation}
and the second component admits the characterisation 
\begin{equation}\label{eq:def-Y-0}
\bY_0:=\{\btau \in \bY:\quad \bdiv\, \btau = \bzero\}=\{\btau \in \widetilde{\bY}:\quad \bdiv\, \btau = \bzero \qan \btau=\btau^{\tt t} \}.
\end{equation}

On the other hand, we note that $b_2$ satisfies the inf-sup condition
\begin{equation}\label{eq:infsup-b2}
\sup_{\bzero\neq(\ube,\btau)\in \bX\times\bY } \frac{b_2((\ube,\btau),\bv)}{\|(\ube,\btau)\|_{\bX\times\bY}} \geq \beta_{b_2}\,\|\bv\|_{\bZ}
\quad \forall\,\bv\in \bZ.
\end{equation}
This is a well-known result, proven by means of the wellposed auxiliary problem of finding, for a given $\bv\in\bZ$, the unique $\by \in \bH^1_0(\Omega)$ such that 
\[-\bdiv[\boldsymbol{\epsilon}(\by)] = \bv \quad \text{in }\Omega; \qquad 
\by = \bzero \quad \text{on $\Gamma$},\]
and then constructing $\hat{\btau} = \boldsymbol{\epsilon}(\by)$ which clearly belongs to $\bY$ and, moreover, it satisfies $\|\hat{\btau}\|_{\bY} \leq C\|\bv\|_{\bZ}$. 
Similarly, 
from  \cite[Section 3.4.3.1]{gatica14} we have that  there exists $\beta_{\tilde{b}_2}>0$ such that 
\begin{equation}\label{eq:infsup-hat-b2}
\sup_{\bzero\neq(\ube,\btau)\in \widetilde{\bX}\times\widetilde{\bY} } \frac{\tilde{b}_2((\ube,\btau),\ubv)}{\|(\ube,\btau)\|_{\widetilde{\bX}\times\widetilde{\bY}}} \geq \beta_{\tilde{b}_2}\,\|\ubv\|_{\widetilde{\bZ}}
\quad \forall\,\ubv\in \widetilde{\bZ}.
\end{equation}
In addition, we note that for all $\btau\in \bY_0$, it suffices to take  $\be=\btau$ to easily arrive at 
\begin{equation}\label{eq:infsup-b1-hat-b1}
\ds \sup_{\bzero\neq\ube\in \bX}  \frac{b_1(\ube,\btau)}{\|\ube\|_{\bX}} \geq \|\btau\|_{\bY}
\quad \forall\,\btau\in \bY_0 \qan  \sup_{\bzero\neq\ube\in \widetilde{\bX}}  \frac{\tilde{b}_1(\ube,\btau)}{\|\ube\|_{\widetilde{\bX}}} \geq \|\btau\|_{\widetilde{\bY}}
\quad \forall\,\btau\in \bY_0. 
\end{equation}

\section{Existence and uniqueness of weak solution}\label{sec:wellp} 
\subsection{Preliminaries}\label{sec:monotone}
We stress that if the permeability $\kappa$ is a positive constant $\kappa = \kappa_0$ or a space-dependent uniformly bounded scalar field $\kappa(\bx)$ in $L^\infty(\Omega)$, or a positive definite matrix $\kappa = \mathbb{K}(\bx)$, then the variational forms $a(\cdot,\cdot)$ and $\tilde{a}(\cdot,\cdot)$ are  bilinear forms bounded and coercive in $\bX$ and $\widetilde{\bX}$, respectively. In this case the systems \eqref{eq:weak1} and \eqref{eq:weak2} are linear twofold saddle-point problems which are uniquely solvable, thanks to the properties of the bilinear forms $\tilde{b}_i(\cdot,\cdot),b_i(\cdot,\cdot)$ and owing to, e.g., \cite[Theorem 3.1]{howell2011inf}.

On the other hand, if the permeability 
in the variational forms $a(\cdot,\cdot)$, $\tilde{a}(\cdot,\cdot)$ induces  monotone and Lipschitz-continuous nonlinear operators, i.e., 
\begin{gather*}
    \bA:\bX\to \bX', \quad \ubd\mapsto \bA(\ubd), \quad \langle \bA(\ubd),\ube\rangle := a(\ubd,\ube),\\
\tilde{\bA}:\widetilde{\bX}\to \widetilde{\bX}', \quad \ubd\mapsto \tilde{\bA}(\ubd), \quad \langle \tilde{\bA}(\ubd),\ube\rangle := \tilde{a}(\ubd,\ube),
\end{gather*}
with  
    \begin{align*}
        |\langle \bA(\ubd_1)-\bA(\ubd_2),\ubd_1-\ubd_2\rangle | & \gtrsim \| \ubd_1-\ubd_2\|^2_{\bX}, \qquad  \| \bA(\ubd_1)-\bA(\ubd_2)\|_{\bX'}  \lesssim \|  \ubd_1-\ubd_2\|_{\bX}, \\
         |\langle \tilde{\bA}(\ubd_1)-\tilde{\bA}(\ubd_2),\ubd_1-\ubd_2\rangle | & \gtrsim \| \ubd_1-\ubd_2\|^2_{\widetilde{\bX}}, \qquad \| \tilde{\bA}(\ubd_1)-\tilde{\bA}(\ubd_2)\|_{\widetilde{\bX}'}  \lesssim \|  \ubd_1-\ubd_2\|_{\widetilde{\bX}}, 
    \end{align*} 
    then the systems \eqref{eq:weak1} and \eqref{eq:weak2} are nonlinear twofold saddle-point problems, which are uniquely solvable thanks to the properties of the bilinear forms $\tilde{b}_i(\cdot,\cdot),b_i(\cdot,\cdot)$ and a direct application of \cite[Lemma 2.1]{gatica03}. 

However, and as discussed in \cite{bociu2022weak,bociu2021nonlinear,van2023mathematical}, some of the typical nonlinearities assumed by $\kappa$ \eqref{eq:kappa} do not guarantee monotonicity of the nonlinear operators $\bA,\tilde{\bA}$.


Note, for example, that in \cite{van2023mathematical} the authors ask that $\kappa$ (they only consider it a function of the dilation $\tr\bd$) is such that 
\[ \kappa \in C^1(\Omega), \quad \kappa(0)>0, \quad \kappa'>0, \]
 in \cite{gaspar2016numerical} the permeability $\kappa$ depends only on the fluid pressure $p$ and it is assumed  that 
\[0< k_0 \leq  \kappa(s) \leq k_1 \qquad \forall s \in \mathbb{R}^+,\]
and in \cite{fu2020constraint} a similar uniform boundedness is assumed even if the permeability depends on both pore pressure and the symmetric strain.  
In our case, for sake of the analysis in this section, we allow the permeability $\kappa(\ubd)=\kappa(\bd,p)$ to be anisotropic but still require that it is a uniformly positive definite second-order tensor in $\bbL^\infty(\Omega)$, and  Lipschitz continuous in $p\in \rH^1(\Omega)$. That is, there exist positive constants $\kappa_1,\kappa_2$ such that 
	\begin{equation}
\label{prop-kappa}	\kappa_1|\bv|^2 \leq \bv^{\tt t}\kappa(\cdot,\cdot)\bv, 
\qquad 
\|\kappa(\cdot,p_1) - \kappa(\cdot,p_2)\|_{\bbL^\infty(\Omega)} \leq \kappa_2 \|p_1 -p_2\|_{1,\Omega},
\end{equation}
for all $ \bv,\bw \in\mathbb{R}^d\setminus\{\bzero\}$, and for all $p_1,p_2\in \rH^1(\Omega)$. 


\subsection{Definition of a fixed-point operator}
In view of the  {discussion}
in Section~\ref{sec:monotone}, if $\bA$ or $\tilde{\bA}$ (the operators induced by 
the nonlinear weak forms $a(\cdot,\cdot)$ or $\tilde{a}(\cdot,\cdot)$, respectively) are not monotone, then we proceed to define, for a given $r>0$, the following sets 
\begin{equation}\label{eq:set-W}
 \bW := \Big\{ \ubw:=(\bw,s)\in\bX \,:\quad \|\ubw\|_{\bX} \leq r \Big\} \qan \widetilde{\bW} := \Big\{ \ubw:=(\bw,s)\in\widetilde{\bX} \,:\quad \|\ubw\|_{\widetilde{\bX}} \leq r \Big\},   
\end{equation}
which are  closed balls of $\bX$ and $\widetilde{\bX}$, respectively, with centre at the origin and radius $r$. Next, for a fixed $\ubw:=(\bw,s)$ in $\bW$ or $\widetilde{\bW}$, we define the bilinear forms $a_{\ubw}:\bX \times \bX \to \mathbb{R}$ and $\tilde{a}_{\ubw}:\widetilde{\bX} \times \widetilde{\bX} \to \mathbb{R}$ as follows 
\begin{equation}\label{eq:atilde}
\begin{array}{cc}
\ds a_{\ubw}(\ubd,\ube) :=\int_\Omega \C\bd:\be + \int_\Omega \kappa(\ubw) \nabla p\cdot \nabla q + c_0 \int_{\Omega} p\,q +  \alpha \int_{\Omega} q\, \tr\bd-\alpha \int_{\Omega} p\, \tr\be,\quad \forall\,\ubd,\,\ube \in \bX,\\[3ex]
\ds \tilde{a}_{\ubw}(\ubd,\ube) :=\int_\Omega \C\bd:\be + \int_\Omega \kappa(\ubw) \nabla p\cdot \nabla q + c_0 \int_{\Omega} p\,q +  \alpha \int_{\Omega} q\, \tr\bd-\alpha \int_{\Omega} p\, \tr\be,\quad \forall\,\ubd,\,\ube \in \widetilde{\bX}.
\end{array}
\end{equation}
Thanks to the assumptions on the nonlinear permeability, we can infer that these forms are continuous 
\begin{equation}\label{eq:bounded-a}
\begin{array}{cc}
\big|a_{\ubw}(\ubd,\ube) \big|\leq C_a  \|\ubd\|_{\bX} \|\ube\|_{\bX}\qquad \forall\, \ubd,\,\ube \in \bX,\\
\big|\tilde{a}_{\ubw}(\ubd,\ube) \big|\leq C_a  \|\ubd\|_{\widetilde{\bX}} \|\ube\|_{\widetilde{\bX}}\qquad \forall\, \ubd,\,\ube \in \widetilde{\bX},
\end{array}
\end{equation}
with $C_a := \max\{2\mu+d\lambda,\,c_0,\, d\, \alpha,\,\kappa_2\,r\}>0$, as well as coercive over all of $\bX$ and $\widetilde{\bX}$, respectively
\begin{equation}\label{eq:coer-a}
\begin{array}{cc}
a_{\ubw}(\ube,\ube) \geq c_a \|\ube\|^2_{\bX} \qquad \forall\, \ube \in \bX,\\
\tilde{a}_{\ubw}(\ube,\ube) \geq c_a \|\ube\|^2_{\widetilde{\bX}} \qquad \forall\, \ube \in \widetilde{\bX},
\end{array}
\end{equation}
with $c_a := \min\{\kappa_1,\,2\mu,\,c_0\}>0$. 

Then we define the following fixed-point operators 
\begin{equation}\label{def:operator-J}
\begin{array}{cc}
\cJ: \bW\subseteq \bX\to \bX,\quad \ubw\mapsto \cJ(\ubw) := \ubd,\\[2ex]
\tilde{\cJ}: \widetilde{\bW}\subseteq \widetilde{\bX}\to \widetilde{\bX},\quad \ubw\mapsto \tilde{\cJ}(\ubw) := \ubd,
\end{array}
\end{equation}
where given $\ubw=(\bw,s)\in\bW$, $\cJ(\ubw)=\ubd=(\bd,p)\in\bX$ is the first component of the solution of the linearised version of problem \eqref{eq:weak1}: Find $((\ubd,\bsigma),\bu)\in \bX\times\bY\times\bZ$ such that
\begin{equation}\label{eq:weak1-linearized}
\begin{array}{rlll}
\ds a_{\ubw}(\ubd,\ube)  & +\quad b_1(\ube,\bsigma) & \, & = G(\ube), \\ [1ex]
\ds b_1(\ubd,\btau)  & \, & +\quad b_2(\btau,\bu) & =H(\btau), \\ [1ex]
\,  & \qquad b_2(\bsigma,\bv) & \, & = F(\bv),
\end{array}
\end{equation}
for all $((\ube,\btau),\bv)\in (\bX\times\bY)\times\bZ$. On the other hand, given $\ubw=(\bw,s)\in\widetilde{\bW}$, $\widetilde{\cJ}(\ubw)=\ubd=(\bd,p)\in\widetilde{\bX}$ is the first component of the solution of the linearised version of problem \eqref{eq:weak2}: Find $((\ubd,\bsigma),\ubu)\in \widetilde{\bX}\times\widetilde{\bY}\times\widetilde{\bZ}$ such that
\begin{equation}\label{eq:weak2-linearized}
\begin{array}{rlll}
\ds \tilde{a}_{\ubw}(\ubd,\ube)  & +\quad \tilde{b}_1(\ube,\bsigma) & \, & = \widetilde{G}(\ube), \\ [1ex]
\ds \tilde{b}_1(\ubd,\btau)  & \, & +\quad \tilde{b}_2(\btau,\ubu) & =\widetilde{H}(\btau), \\ [1ex]
\,  & \qquad \tilde{b}_2(\bsigma,\ubv) & \, & = \widetilde{F}(\ubv),
\end{array}
\end{equation}
for all $((\ube,\btau),\ubv)\in (\widetilde{\bX}\times\widetilde{\bY})\times\widetilde{\bZ}$.

It is clear that $((\ubd,\bsigma),\bu)$ is a solution to \eqref{eq:weak1} if and only if $\ubd$ satisfies $\cJ(\ubd) = \ubd$, and consequently, the wellposedness of \eqref{eq:weak1} is equivalent to the unique solvability of the fixed-point problem: Find $\ubd\in \bW$ such that
\begin{equation}\label{eq:fixed-point-problem-1}
\cJ(\ubd) = \ubd.
\end{equation}
Similarly, the tuple $((\ubd,\bsigma),\ubu)$ is a solution to \eqref{eq:weak2} if and only if $\ubd$ satisfies $\widetilde{\cJ}(\ubd) = \ubd$, and consequently, the wellposedness of \eqref{eq:weak2} is equivalent to the unique solvability of the fixed-point problem: Find $\ubd\in \widetilde{\bW}$ such that
\begin{equation}\label{eq:fixed-point-problem-2}
\widetilde{\cJ}(\ubd) = \ubd.
\end{equation}

In this way, in what follows we focus on proving the unique solvability of \eqref{eq:fixed-point-problem-1} and \eqref{eq:fixed-point-problem-2}. 
According to the definition of $\cJ$ and $\widetilde{\cJ}$ (cf. \eqref{def:operator-J}), it is clear that  proving that these operators are well-defined  amounts to prove that problems \eqref{eq:weak1-linearized} and \eqref{eq:weak2-linearized}, respectively, are wellposed.

With that in mind, let us define the bilinear forms $A_{\ubw} : (\bX\times\bY_0)\times (\bX\times\bY_0)\to \bbR$ and $\tilde{A}_{\ubw} : (\widetilde{\bX}\times\bY_0)\times (\widetilde{\bX}\times\bY_0)\to \bbR$ as
\begin{subequations}
\begin{equation}\label{eq:A}
A_{\ubw}( (\ubd,\bsigma),(\ube,\btau) ):= 
a_{\ubw}(\ubd,\ube) + b_1(\ube,\bsigma) + b_1(\ubd,\btau) ,
\end{equation}
and
\begin{equation}\label{eq:A-tilde}
\tilde{A}_{\ubw}( (\ubd,\bsigma),(\ube,\btau) ):= 
\tilde{a}_{\ubw}(\ubd,\ube) + \tilde{b}_1(\ube,\bsigma) + \tilde{b}_1(\ubd,\btau) ,
\end{equation}\end{subequations}
respectively, and we state the unique solvability of the linearised problems \eqref{eq:weak1-linearized} and \eqref{eq:weak2-linearized}, depending on a smallness of data assumption, as follows. 

\begin{lemma}\label{lem:well-def-J}
Given $r>0$, let us assume that 
\begin{equation}\label{eq:assumption-J}
\dfrac{\gamma_1}{r} \big( 
 \|g\|_{0,\Omega}+ \|r_\Gamma \|_{-1/2,\Gamma} + \|\bu_\Gamma\|_{1/2,\Gamma} + \|\fb\|_{0,\Omega}   \big)\leq 1,
\end{equation}
where
\begin{equation}\label{eq:def-gamma}
\gamma_1:=\dfrac{(C_{a} + 1 + \beta_2 +\gamma_2)^2}{\beta^2_2\, \gamma_2} \qan \gamma_2:=\dfrac{(C_{a}+2)(c_{a}+1+C_{a})}{c_{a}}.
\end{equation}
Then, for a given $\ubw\in\bW$ (cf. \eqref{eq:set-W}),  there exists a unique $\ubd\in\bW$ such that $\cJ(\ubw) = \ubd$.
\end{lemma}
\begin{proof}
From the properties of $a_{\ubw}$ and $b_1$, \eqref{eq:bounded-a}, \eqref{eq:coer-a}, \eqref{eq:bounded-b1-b2} and \eqref{eq:infsup-b1-hat-b1}, we have that the bilinear form $A_{\ubw}$ induces an invertible operator on the kernel of the bilinear form $b_2$, $\bX\times\bY_0$ (cf. \eqref{eq:kernel-b2-hat-b2}). Then, from the inf-sup condition of $b_2$ \eqref{eq:infsup-b2}, and a straightforward application of the Babu\v ska--Brezzi theory we have that there exists a unique $((\ubd,\bsigma),\bu)\in 
 (\bX\times\bY)\times\bZ$ solution to \eqref{eq:weak1-linearized}, or equivalently, the existence of a unique $\ubd\in\bX$ such that $\cJ(\ubw) = \ubd$. Finally, from \cite[Proposition 2.36]{ernguermond},  together with \eqref{des:bound-F-H-G}, we readily obtain that
\begin{equation}\label{eq:J-from-W-to-W}
\begin{array}{ll}
\|(\ubd,\bsigma),\bu\|_{(\bX\times\bY)\times\bZ} \\[2ex]
\qquad \ds \leq \gamma_1\, \sup_{\bzero\neq((\ube,\btau),\bv)\in 
 (\bX\times\bY)\times\bZ } \frac{A_{\ubw}( (\ubd,\bsigma),(\ube,\btau) ) + b_2((\ubd,\bsigma),\bv) + b_2((\ube,\btau),\bu)}{\|(\ube,\btau),\bv\|}\\[4ex]
\qquad \ds \leq \gamma_1\, \big( 
 \|g\|_{0,\Omega}+ \|r_\Gamma \|_{-1/2,\Gamma} + \|\bu_\Gamma\|_{1/2,\Gamma} + \|\fb\|_{0,\Omega}   \big),
\end{array}
\end{equation}
and after invoking assumption \eqref{eq:assumption-J}, the bounds above imply that $\ubd$ belongs to $\bW$, therefore completing the proof.
\end{proof}

\begin{lemma}\label{lem:well-def-hat-J}
Given $r>0$, assume that 
\begin{equation}\label{eq:assumption-hat-J}
\dfrac{\tilde\gamma_1}{r} \big( 
 \|g\|_{0,\Omega}+ \|r_\Gamma \|_{-1/2,\Gamma} + \|\bu_\Gamma\|_{1/2,\Gamma} + \|\fb\|_{0,\Omega}   \big)\leq 1,
\end{equation}
where
\begin{equation}\label{eq:def-hat-gamma}
\tilde\gamma_1:=\dfrac{(C_{a} + 1 + \tilde{\beta}_2 +\gamma_2)^2}{\tilde{\beta}^2_2\, \gamma_2},
\end{equation}
and with $\gamma_2$ defined in \eqref{eq:def-gamma}. Then, given $\ubw\in\widetilde{\bW}$ (cf. \eqref{eq:set-W}) there exists a unique $\ubd\in\widetilde{\bW}$ such that $\tilde{\cJ}(\ubw) = \ubd$.
\end{lemma}
\begin{proof}
The proof follows using the same steps employed to prove Lemma~\ref{lem:well-def-J}.
\end{proof}
\subsection{Wellposedness of the continuous problem}
Here, we provide the main result of this section, namely, the existence and uniqueness of solution of the nonlinear problems \eqref{eq:weak1} and \eqref{eq:weak2}. This result is established in the following theorems. 
\begin{theorem}\label{theorem:unique-solution-weak1}
Let $\fb \in \bL^2(\Omega)$, $g \in L^2(\Omega)$, $\bu_\Gamma \in \bH^{1/2}(\Gamma)$ and $r_\Gamma \in \rH^{-1/2}(\Gamma)$ such that
\begin{equation}\label{eq:assumption-J-2}
\ds \dfrac{\gamma_1}{r}\max\{\gamma_1\,\kappa_2\,r,1\}\, \big( 
 \|g\|_{0,\Omega}+ \|r_\Gamma \|_{-1/2,\Gamma} + \|\bu_\Gamma\|_{1/2,\Gamma} + \|\fb\|_{0,\Omega}   \big) < 1,
\end{equation}
where $\gamma_1$ is defined in \eqref{eq:def-gamma}.
Then, the operator $\cJ$ (cf. \eqref{def:operator-J}) has a unique fixed point $\ubd\in\bW$. Equivalently, the problem \eqref{eq:weak1} has a unique solution $((\ubd,\bsigma),\bu)\in (\bX\times\bY)\times\bZ$ with $\ubd\in\bW$. In addition, we have the following continuous dependence on data
%
\begin{equation}\label{eq:stability-weak1}
\|(\ubd,\bsigma),\bu\| \lesssim  
 \|g\|_{0,\Omega}+ \|r_\Gamma \|_{-1/2,\Gamma} + \|\bu_\Gamma\|_{1/2,\Gamma} + \|\fb\|_{0,\Omega}.
\end{equation}
\end{theorem}
\begin{proof}
We begin by recalling from the previous analysis that assumption \eqref{eq:assumption-J-2} ensures the well-definiteness of $\cJ$. Now, let $\ubw_{1}=(\bw_1,s_1)$, $\ubw_{2}=(\bw_2,s_2)$, $\ubd_{1}=(\bd_1,p_1)$, $\ubd_{2}=(\bd_2,p_2)$ $\in\bW$, be such that $\cJ(\ubw_{1})=\ubd_{1}$ and $\cJ(\ubw_{2})=\ubd_{2}$. According to the definition of $\cJ$ (cf. \eqref{eq:weak1-linearized}), it follows that there exist 
$(\bsigma_1,\bu_1)$, $(\bsigma_2,\bu_2)$ $\in\bY\times\bZ$, such that for all $((\ube,\btau),\bv)\in (\bX\times\bY)\times\bZ$, there hold
\begin{equation*}
\begin{array}{cc}
A_{\ubw_1}( (\ubd_1,\bsigma_1),(\ube,\btau) ) + b_2((\ube,\btau),\bu_1) + b_2((\ubd_1,\bsigma_1),\bv) = G(\ube) + H(\btau) + F(\bv),\\[2ex]
A_{\ubw_2}( (\ubd_2,\bsigma_2),(\ube,\btau) ) + b_2((\ube,\btau),\bu_2) + b_2((\ubd_2,\bsigma_2),\bv) = G(\ube) + H(\btau) + F(\bv).
\end{array}
\end{equation*}
Then, subtracting both equations, adding and subtracting suitable terms, we easily arrive at
\begin{equation}\label{eq:A1-A2}
\begin{array}{ll}
A_{\ubw_1}( (\ubd_1-\ubd_2,\bsigma_1-\bsigma_2),(\ube,\btau) ) + b_2((\ube,\btau),\bu_1-\bu_2) + b_2((\ubd_1-\ubd_2,\bsigma_1-\bsigma_2),\bv) \\[2ex]
\ds \qquad = \int_\Omega (\kappa(\ubw_2) - \kappa(\ubw_1) \nabla p_2\cdot \nabla q.
\end{array}
\end{equation}
Therefore, recalling that $\ubw_1\in\bW$, we can use the latter identity, the bound \eqref{eq:J-from-W-to-W}, and invoke the Lipschitz continuity of $\kappa$ (cf. \eqref{prop-kappa}), to obtain
\begin{equation*}
\begin{array}{ll}
\ds\|\ubd_{1} - \ubd_{2}\|_\bX  \,\leq\, 
\|(\ubd_{1} - \ubd_{2}, \bsigma_1 -\bsigma_2 ),\bu_1 - \bu_2\|_{ (\bX\times\bY) \times \bZ} \\[2ex]
\ds \leq \gamma_1 \sup_{\substack{\bzero\neq((\ube,\btau),\bv) \\ \quad \in (\bX\times\bY)\times\bZ} } \frac{A_{\ubw_1}( (\ubd_1-\ubd_2,\bsigma_1-\bsigma_2),(\ube,\btau) )  + b_2((\ube,\btau),\bu_1-\bu_2) + b_2((\ubd_1-\ubd_2,\bsigma_1-\bsigma_2),\bv) }{\|(\ube,\btau),\bv\|}\\[4ex]
\ds = \gamma_1\, \sup_{\substack{\bzero\neq((\ube,\btau),\bv) \\ \quad \in (\bX\times\bY)\times\bZ} } \dfrac{\ds\int_\Omega (\kappa(\ubw_2) - \kappa(\ubw_1) \nabla p_2\cdot \nabla q}{\|(\ube,\btau),\bv\|}\\[4ex]
 \ds \leq \gamma_1\,\kappa_2 \,\|s_2 -s_1\|_{1,\Omega} \|\nabla p_2\|_{0,\Omega}\, ,
\end{array}
\end{equation*}
which, together with the fact that $\ubd_{2}=(\bd_2,p_2)\in \bW$ and the estimate \eqref{eq:J-from-W-to-W}, implies that 
\begin{equation*}
\begin{array}{ll}
\|\cJ(\ubw_{1})- \cJ(\ubw_{2})\|_\bX \,=\, \|\ubd_{1} - \ubd_{2}\|_\bX \\[2ex]
\qquad\ds\leq \gamma_1^2\,\kappa_2\, \big( 
 \|g\|_{0,\Omega}+ \|r_\Gamma \|_{-1/2,\Gamma} + \|\bu_\Gamma\|_{1/2,\Gamma} + \|\fb\|_{0,\Omega}   \big) \|s_1 -s_2\|_{1,\Omega}\\[2ex]
 \qquad\ds\leq \gamma_1^2\,\kappa_2\,  \big( 
 \|g\|_{0,\Omega}+ \|r_\Gamma \|_{-1/2,\Gamma} + \|\bu_\Gamma\|_{1/2,\Gamma} + \|\fb\|_{0,\Omega}   \big) \|\ubw_1 - \ubw_2\|_{\bX}.
\end{array}
\end{equation*}
The latter bound, in combination with the assumption \eqref{eq:assumption-J-2} and the Banach fixed-point theorem, implies that $\cJ$ has a unique fixed point in $\bW$.  Equivalently, this result yields that there exists a unique $((\ubd,\bsigma),\bu)\in (\bX\times\bY)\times\bZ$ solution to \eqref{eq:weak1}. Finally, estimate \eqref{eq:stability-weak1} is obtained analogously to \eqref{eq:J-from-W-to-W}, which completes the proof.
\end{proof}
\begin{theorem}\label{theorem:unique-solution-weak2}
Let $\fb \in \bL^2(\Omega)$, $g \in L^2(\Omega)$, $\bu_\Gamma \in \bH^{1/2}(\Gamma)$ and $r_\Gamma \in \rH^{-1/2}(\Gamma)$ such that
\begin{equation}\label{eq:assumption-hat-J-2}
\ds \dfrac{\tilde\gamma_1}{r}\max\{\tilde\gamma_1\,\kappa_2\,r,1\}\, \big( 
 \|g\|_{0,\Omega}+ \|r_\Gamma \|_{-1/2,\Gamma} + \|\bu_\Gamma\|_{1/2,\Gamma} + \|\fb\|_{0,\Omega}   \big) < 1,
\end{equation}
where $\tilde\gamma_1$ is defined in \eqref{eq:def-hat-gamma}.
Then, the operator $\tilde{\cJ}$ (cf. \eqref{def:operator-J}) has a unique fixed point $\ubd\in\widetilde{\bW}$. Equivalently, the problem \eqref{eq:weak2} has a unique solution $((\ubd,\bsigma),\ubu)\in \widetilde{\bX}\times\widetilde{\bY}\times\widetilde{\bZ}$ with $\ubd\in\widetilde{\bW}$. In addition, there 
holds 
\begin{equation}\label{eq:stability-weak2}
\|(\ubd,\bsigma),\ubu\| \lesssim 
 \|g\|_{0,\Omega}+ \|r_\Gamma \|_{-1/2,\Gamma} + \|\bu_\Gamma\|_{1/2,\Gamma} + \|\fb\|_{0,\Omega}.
\end{equation}
\end{theorem}
\begin{proof}
The proof follows using the same steps employed to prove Theorem~\ref{theorem:unique-solution-weak1}.
\end{proof}

\section{Finite element discretisation}\label{sec:FE}
Let us consider a  regular partition $\mathcal T_h$ of $\bar\Omega$ made up of triangles $K$ 
(in $\mathbb{R}^2$) or tetrahedra $K$ (in $\mathbb{R}^3$) of diameter $h_K$, and denote the mesh size by 
$h := \max\{ h_K: \ K \in \mathcal T_h\}$. We will start by defining finite-dimensional subspaces $\bX_h$, $\bY_h$, $\bZ_h$, $\widetilde{\bX}_h$, $\widetilde{\bY}_h$, $\widetilde{\bZ}_h$ of the functional spaces encountered before. 

Given an integer $\ell \ge 0$ and $K \in \mathcal{T}_h$, we first let $\mathrm{P}_\ell(K)$ be the space of polynomials of degree $\leq \ell$ defined on $K$, whose vector and tensor versions are denoted $\bP_\ell(K) \,:=\, [\mathrm{P}_\ell(K)]^d$ and $\mathbb{P}_\ell(K)
\,=\,[\mathrm{P}_\ell(K)]^{d \times d}$, respectively. Also, we let $\mathbf{RT}_\ell(K) \,:=\, \bP_\ell(K) \oplus \mathrm{P}_\ell(K)\,\bx$ be the local Raviart--Thomas space of order $\ell$ defined on $K$, where $\bx$ stands for a generic vector in $\bbR^d$. 

\subsection{Finite element spaces and definition of the Galerkin scheme}\label{sec:spaces}
Irrespective of the discrete spaces used for the other unknowns, for fluid pressure we take Lagrangian elements as follows 
\[\bX_{2,h}:= \left\{ q_h \in \mathrm{C}(\overline{\Omega}) :\quad  q_h|_{K}\in \textrm{P}_{k+1}(K)\;\; \quad \forall\, K\in \mathcal{T}_h\right\}.\]

For each $K\in \mathcal{T}_h$ we consider the bubble space of order $k$, defined as 
\[\mathbf{B}_k(K):=
 \begin{cases}
\mathbf{curl}^{\mathrm{t}}(b_K\mathrm{P}_k(K))&\textrm{in}\quad\mathbb{R}^2,\\
\nabla \times (b_K\mathbf{P}_k(K))&\textrm{in} \quad \mathbb{R}^3,
\end{cases}\]
{where $b_K$ is a suitably normalised cubic polynomial on $K$, which vanishes on the boundary of $K$ (see \cite{ernguermond}).}

Arnold--Winther finite elements are defined in \cite{arnold-winther-0}   for $k\geq 1$  and   for the 2D case. {The lowest-order conforming space for poroelastic stress (and here also for strain) consists of piecewise $\mathbb{P}_2$ tensors enriched with cubic shape functions, and piecewise $\mathbb{P}_1$ vectors for displacement}:  
\begin{align}
{\bY}_{h}&:=\left\{{\btau}_h \in \bbHsym:\  {\btau}_{h}|_{K}\in \mathbb{P}_{k+2}(K)\ \text{and} \ \bdiv\,\btau_h|_K \in \mathbf{P}_{k}(K)\quad \forall\, K\in \mathcal{T}_h\right\}, \nonumber\\\label{fe:aw}
{\bZ}_{h}&:= \left\{\bv_h \in \bL^2(\Omega): \  \bv_h|_{K}\in \textbf{P}_{k}(K)  \quad \forall\, K\in \mathcal{T}_h\right\}.
\end{align}
Note that a non-conforming version is also given in \cite{arnold-winther-0} but it gives an unbalanced approximation error for displacement and stress and we therefore keep only the conforming version. 
An appropriate interpolation operator   (bounded, with suitable approximability and  commutation properties) is constructed in \cite{arnold-winther-0}, which thanks to Fortin's Lemma (cf. \cite[Lemma 2.6]{gatica14}), imply a discrete inf-sup condition for $b_2(\cdot,\cdot)$ (see also \cite{howell2011inf}).

As announced, the following space for discrete strains is considered 
\begin{equation}\label{eq:X-1-h}
  \bX_{1,h}:=\left\{{\be}_h \in \bbLsym:\;\; {\btau}_{h}|_{K}\in \mathbb{P}_{k+2}(K)\ \text{and} \ \bdiv\,\btau_h|_K \in \mathbf{P}_{k}(K)\quad \forall\, K\in \mathcal{T}_h\right\}.  
\end{equation}
Then, defining the product space $\bX_h:=\bX_{1,h}\times\bX_{2,h}$, we note that the finite element subspaces  $(\bX_h\times\bY_h)\times\bZ_h$ are inf-sup stable for the bilinear form $b_2$ (cf. \cite{arnold-winther-0}) 
\begin{equation}\label{eq:infsup-b2-h}
\sup_{\bzero\neq(\ube_h,\btau_h)\in \bX_h\times\bY_h } \frac{b_2((\ube_h,\btau_h),\bv_h)}{\|(\ube_h,\btau_h)\|_{\bX\times\bY}} \geq \beta^*_{b_2}\,\|\bv_h\|_{\bZ}
\quad \forall\,\bv_h\in \bZ_h.
\end{equation}
In addition, it is straightforward to see that the kernel of the bilinear form $b_2$ can be characterised by
\begin{equation}\label{eq:kernel-b2-h}
\bX_h\times\bY_{h,0},\quad \text{with}\quad \bY_{h,0} = \{\btau_h \in \bY_h:\ \bdiv\, \btau_h = \bzero \},
\end{equation}
and, similarly to \eqref{eq:infsup-hat-b1-h}, using that $\bY_{h,0}\subset\bX_{1,h}$, for all $\btau_h\in\tilde{\bY}_{h,0}$ we can take $\be_h=\btau_h$ to conclude that $b_1$ satisfies the inf-sup condition 
\begin{equation}\label{eq:infsup-b1-h}
\ds \sup_{\bzero\neq\ube_h\in \bX_h}  \frac{b_1(\ube_h,\btau_h)}{\|\ube_h\|_{\bX}} \geq \|\btau_h\|_{\bY}
\quad \forall\,\btau_h\in \bY_{0,h}. 
\end{equation}

The Galerkin scheme associated with the weak formulation \eqref{eq:weak1} consists in finding $((\ubd_h,\bsigma_h),\bu_h)\in (\bX_h\times\bY_h)\times\bZ_h$ such that 
\begin{equation}\label{eq:weak1-h}
\begin{array}{rlll}
\ds a(\ubd_h,\ube_h)  & +\quad b_1(\ube_h,\bsigma_h) & \, & = G(\ube_h), \\ [1ex]
\ds b_1(\ubd_h,\btau_h)  & \, & +\quad b_2((\ube_h,\btau_h),\bu_h) & =H(\btau_h), \\ [1ex]
\,  & \qquad b_2((\ubd_h,\bsigma_h),\bv_h) & \, & = F(\bv_h),
\end{array}
\end{equation}
for all $((\ube_h,\btau_h),\bv_h)\in (\bX_h\times\bY_h)\times\bZ_h$, with $\ubd_h=(\bd_h,p_h)$ and $\ube_h=(\be_h,q_h)$. 

Next, 
we recall that the classical PEERS  elements are described in \cite{arnold84}: 
\begin{align}\label{fe:peers}
\widetilde{\bY}_{h}&:=\left\{{\btau}_h \in \mathbb{H}(\mathbf{div},\Omega):\quad {\btau}_{h}|_{K}\in \mathbb{RT}_k(K)\oplus[\mathbf{B}_k(K)]^d\quad \forall\, K\in \mathcal{T}_h\right\},\nonumber \\
\widetilde{\bZ}_{1,h}&:= \left\{\bv_h \in \bL^2(\Omega): \quad \bv_h|_{K}\in \textbf{P}_{k}(K)\quad \forall\, K\in \mathcal{T}_h\right\},\\
\widetilde{\bZ}_{2,h}&:=\left\{\bbeta_h\in \bbLskew\cap \mathbf{C}(\overline{\Omega})\quad \mathrm{and}\quad \bbeta_h|_{K}\in \mathbb{P}_{k+1}(K)\quad \forall\, K\in \mathcal{T}_h\right\},\nonumber
\end{align}
and are inf-sup stable for the bilinear form $\tilde{b}_2$. 
In addition, and according to \cite{gatica2013priori}, they are inf-sup stable together with the space  
\begin{equation}\label{eq:hat-X-1-h-a}
\widetilde{\bX}_{1,h}:=\left\{{\be}_h \in \mathbb{L}^{2}(\Omega):\;\; {\be}_{h}|_{K}\in \mathbb{P}_k(K)\oplus[\mathbf{B}_k(K)]^d
\quad \forall\, K\in \mathcal{T}_h\right\},
\end{equation}
with respect to $b_1$ (see also \cite{gatica03,gomez23} where also the deviatoric part of the bubble functions is used in the enrichment).

Moreover, Arnold--Falk--Winther finite elements are in \cite{arnold-2007}: 
\begin{align}
\widetilde{\bY}_{h}&:=\left\{{\btau}_h \in \mathbb{H}(\mathbf{div},\Omega):\  {\btau}_{h}|_{K}\in \mathbb{BDM}_{k+1}(K)\quad \forall\, K\in \mathcal{T}_h\right\}, \nonumber\\\label{fe:afw}
\widetilde{\bZ}_{1,h}&:= \left\{\bv_h \in \bL^2(\Omega): \  \bv_h|_{K}\in \textbf{P}_{k}(K)  \quad \forall\, K\in \mathcal{T}_h\right\},\\ 
\widetilde{\bZ}_{2,h}&:=\left\{\bbeta_h\in \bbLskew:\  \bbeta_h|_{K}\in \mathbb{P}_{k}(K) \quad \forall\, K\in \mathcal{T}_h\right\},
\nonumber
\end{align}
and, together with  the space 
\begin{equation}\label{eq:hat-X-1-h-b}
\widetilde{\bX}_{1,h}:=\left\{{\be}_h \in \mathbb{L}^{2}(\Omega):\;\; {\be}_{h}|_{K}\in \mathbb{BDM}_{k+1}(K)\quad \forall\, K\in \mathcal{T}_h\right\},
\end{equation}
they are inf-sup stable with respect to the bilinear forms $\tilde{b}_2$ and $\tilde{b}_1$. 

In what follows, we will develop the analysis with the subspaces defined in \eqref{fe:peers}--\eqref{eq:hat-X-1-h-a}. The analysis for the subspaces defined in \eqref{fe:afw}--\eqref{eq:hat-X-1-h-b} is conducted  in an analogous manner.

Let us denote the product spaces $\widetilde{\bX}_h:=\widetilde{\bX}_{1,h}\times\bX_{2,h}$ and $\widetilde{\bZ}_h:=\widetilde{\bZ}_{1,h}\times\widetilde{\bZ}_{2,h}$, and note that the finite element subspaces  $\widetilde{\bX}_h\times\widetilde{\bY}_h\times\widetilde{\bZ}_h$ are inf-sup stable for the bilinear form $\tilde{b}_2$ (cf. \cite[Section 4.5]{gatica14}) 
\begin{equation}\label{eq:infsup-hat-b2-h}
\sup_{\bzero\neq(\ube_h,\btau_h)\in \widetilde{\bX}_h\times\widetilde{\bY}_h } \frac{\tilde{b}_2((\ube_h,\btau_h),\ubv_h)}{\|(\ube_h,\btau_h)\|_{\widetilde{\bX}\times\widetilde{\bY}}} \geq \tilde{\beta}^*_{b_2}\,\|\ubv_h\|_{\widetilde{\bZ}}
\quad \forall\,\ubv_h\in \widetilde{\bZ}_h.
\end{equation}
In addition, it is straightforward to see that the kernel of the bilinear form $\tilde{b}_2$ can be characterised by
\begin{equation}\label{eq:kernel-hat-b2-h}
\widetilde{\bX}_h\times\widetilde{\bY}_{h,0},\quad \text{with}\quad \widetilde{\bY}_{h,0} = \{\btau_h \in \widetilde{\bY}_h:\ \bdiv\, \btau_h = \bzero \qan \btau=\btau^{\tt t}\},
\end{equation}
and, for all $\btau_h\in\widetilde{\bY}_{h,0}$, it is clear that $\btau_h\in\widetilde{\bX}_{1,h}$, then we can take $\be_h=\btau_h$, and $\tilde{b}_1$ satisfy the inf-sup condition 
\begin{equation}\label{eq:infsup-hat-b1-h}
\ds \sup_{\bzero\neq\ube_h\in \widetilde{\bX}_h}  \frac{\tilde{b}_1(\ube_h,\btau_h)}{\|\ube_h\|_{\widetilde{\bX}}} \geq \|\btau_h\|_{\widetilde{\bY}}
\quad \forall\,\btau_h\in \widetilde{\bY}_{0,h}. 
\end{equation}

Finally, the scheme associated with the weak formulation \eqref{eq:weak2} consists in finding $((\ubd_h,\bsigma_h),\ubu_h)\in (\widetilde{\bX}_h\times\widetilde{\bY}_h)\times\widetilde{\bZ}_h$ such that 
\begin{equation}\label{eq:weak2-h}
\begin{array}{rlll}
\ds \tilde{a}(\ubd_h,\ube_h)  & +\quad \tilde{b}_1(\ube_h,\bsigma_h) & \, & = \widetilde{G}(\ube_h), \\ [1ex]
\ds \tilde{b}_1(\ubd_h,\btau_h)  & \, & +\quad \tilde{b}_2((\ube_h,\btau_h),\ubu_h) & =\widetilde{H}(\btau_h), \\ [1ex]
\,  & \qquad \tilde{b}_2((\ubd_h,\bsigma_h),\ubv_h) & \, & = \widetilde{F}(\ubv_h),
\end{array}
\end{equation}
for all $((\ube_h,\btau_h),\ubv_h)\in (\widetilde{\bX}_h\times\widetilde{\bY}_h)\times\widetilde{\bZ}_h$, with $\ubd_h=(\bd_h,p_h)$, $\ube_h=(\be_h,q_h)$, $\ubu_h=(\bu_h,\bgamma_h)$ and $\ubv_h=(\bv_h,\bbeta_h)$.

\begin{remark}
Due to the definition of the bilinear forms $b_1$, $b_2$, $\tilde{b}_1$ and $\tilde{b}_2$ (cf. \eqref{eq:def-a-b1-b2} and \eqref{eq:def-hat-a-b1-b2}), the subspaces $\bX_{1,h}$ and $\widetilde{\bX}_{1,h}$, defined in \eqref{eq:X-1-h} and \eqref{eq:hat-X-1-h-a} (or \eqref{eq:hat-X-1-h-b}), can be chosen in many ways; it is sufficient that $\bdiv\, \bY_h \subseteq \bX_{1,h}$ and $\bdiv\, \widetilde{\bY}_h \subseteq \widetilde{\bX}_{1,h}$, with this, the inf-sup condition of $b_1$ and $\tilde{b}_1$ in the kernel of $b_2$ and $\tilde{b}_2$, respectively, can be ensured. For example, we can take the space  $\bX_{1,h}:=\left\{{\be}_h \in \bbLsym:\;\; {\btau}_{h}|_{K}\in \mathbb{P}_{k+2}(K)\  \forall\, K\in \mathcal{T}_h\right\}$, to approximate the strain tensor in the first Galerkin scheme, but it is more expensive.
\end{remark}

\subsection{Unique solvability of the discrete problem}
In this section we analyse the Galerkin schemes \eqref{eq:weak1-h} and \eqref{eq:weak2-h}. We want to emphasise that the analysis of wellposedness can be easily accomplished by applying the results obtained for the continuous problem to the discrete scenario, which is why most of the specific details can be excluded.

Firstly, and 
similarly to the continuous case, we define the following sets 
\begin{equation}\label{eq:set-W-h}
 \bW_h := \Big\{ \ubw_h:=(\bw_h,s_h)\in\bX_h \,:\; \|\ubw_h\|_{\bX} \leq r \Big\}, \quad \widetilde{\bW}_h := \Big\{ \ubw_h:=(\bw_h,s_h)\in\widetilde{\bX}_h \,:\; \|\ubw_h\|_{\widetilde{\bX}} \leq r \Big\} .  
\end{equation}
Next, for a fixed $\ubw_h:=(\bw_h,s_h)$ in $\bW_h$ or $\widetilde{\bW}_h$, we have that the bilinear forms $a_{\ubw}$ and $\tilde{a}_{\ubw}$ defined in \eqref{eq:atilde}, satisfy:
\begin{equation}\label{eq:coer-a-h}
\begin{array}{cc}
a_{\ubw_h}(\ube_h,\ube_h) \geq c_a \|\ube_h\|^2_{\bX} \qquad \forall\, \ube_h \in \bX_h,\\
\tilde{a}_{\ubw_h}(\ube_h,\ube_h) \geq c_a \|\ube_h\|^2_{\widetilde{\bX}} \qquad \forall\, \ube_h \in \widetilde{\bX}_h.
\end{array}
\end{equation}
Then, and again analogously to the continuous case, we define the following fixed-point operators 
\begin{equation}\label{def:operator-J-h}
\begin{array}{cc}
\cJ_h: \bW_h\subseteq \bX_h\to \bX_h,\quad \ubw_h\mapsto \cJ_h(\ubw_h) := \ubd_h,\\[2ex]
\widetilde{\cJ}_h: \widetilde{\bW}_h\subseteq \widetilde{\bX}_h\to \widetilde{\bX}_h,\quad \ubw_h\mapsto \widetilde{\cJ}_h(\ubw_h) := \ubd_h,
\end{array}
\end{equation}
where, given $\ubw_h=(\bw_h,s_h)\in\bW_h$, $\cJ_h(\ubw_h)=\ubd_h=(\bd_h,p_h)\in\bX_h$ is the first component of the solution of the linearised version of problem \eqref{eq:weak1-h}: Find $((\ubd_h,\bsigma_h),\bu_h)\in \bX_h\times\bY_h\times\bZ_h$ such that
\begin{equation}\label{eq:weak1-h-linearized}
\begin{array}{rlll}
\ds a_{\ubw_h}(\ubd_h,\ube_h)  & +\quad b_1(\ube_h,\bsigma_h) & \, & = G(\ube_h), \\ [1ex]
\ds b_1(\ubd_h,\btau_h)  & \, & +\quad b_2(\btau_h,\bu_h) & =H(\btau_h), \\ [1ex]
\,  & \qquad b_2(\bsigma_h,\bv_h) & \, & = F(\bv_h),
\end{array}
\end{equation}
for all $((\ube_h,\btau_h),\bv_h)\in (\bX_h\times\bY_h)\times\bZ_h$. 

On the other hand, given $\ubw_h=(\bw_h,s_h)\in\widetilde{\bW}_h$, $\widetilde{\cJ}_h(\ubw_h)=\ubd_h=(\bd_h,p_h)\in\widetilde{\bX}_h$ is the first component of the solution of the linearised version of problem \eqref{eq:weak2-h}: Find $((\ubd_h,\bsigma_h),\ubu_h)\in \widetilde{\bX}_h\times\widetilde{\bY}_h\times\widetilde{\bZ}_h$ such that
\begin{equation}\label{eq:weak2-h-linearized}
\begin{array}{rlll}
\ds \tilde{a}_{\ubw_h}(\ubd_h,\ube_h)  & +\quad \tilde{b}_1(\ube_h,\bsigma_h) & \, & = \widetilde{G}(\ube_h), \\ [1ex]
\ds \tilde{b}_1(\ubd_h,\btau_h)  & \, & +\quad \tilde{b}_2(\btau_h,\ubu_h) & =\widetilde{H}(\btau_h), \\ [1ex]
\,  & \qquad \tilde{b}_2(\bsigma_h,\ubv_h) & \, & = \widetilde{F}(\ubv_h),
\end{array}
\end{equation}
for all $((\ube_h,\btau_h),\ubv_h)\in (\widetilde{\bX}_h\times\widetilde{\bY}_h)\times\widetilde{\bZ}_h$.

It is clear that $((\ubd_h,\bsigma_h),\bu_h)$ is a solution to \eqref{eq:weak1-h} if and only if $\ubd_h$ satisfies $\cJ_h(\ubd_h) = \ubd_h$, and consequently, the wellposedness of \eqref{eq:weak1-h} is equivalent to the unique solvability of the fixed-point problem: Find $\ubd_h\in \bW_h$ such that
\begin{equation}\label{eq:fixed-point-problem-1-h}
\cJ_h(\ubd_h) = \ubd_h.
\end{equation}
Much in the same way as above, the tuple $((\ubd_h,\bsigma_h),\ubu_h)$ is a solution to \eqref{eq:weak2-h} if and only if $\ubd_h$ satisfies $\widetilde{\cJ}_h(\ubd_h) = \ubd_h$, and consequently, the wellposedness of \eqref{eq:weak2-h} is equivalent to the unique solvability of the fixed-point problem: Find $\ubd_h\in \widetilde{\bW}_h$ such that
\begin{equation}\label{eq:fixed-point-problem-2-h}
\widetilde{\cJ}_h(\ubd_h) = \ubd_h.
\end{equation}
\noindent In this way, in what follows we focus on proving the unique solvability of \eqref{eq:fixed-point-problem-1-h} and \eqref{eq:fixed-point-problem-2-h}. 
\begin{lemma}\label{lem:well-def-J-h}
Given $r>0$, assume that 
\begin{equation}\label{eq:assumption-J-h}
\dfrac{\gamma^*_1}{r} \big( 
 \|g\|_{0,\Omega}+ \|r_\Gamma \|_{-1/2,\Gamma} + \|\bu_\Gamma\|_{1/2,\Gamma} + \|\fb\|_{0,\Omega}   \big)\leq 1,
\end{equation}
where $\gamma^*_1$ is the discrete version of $\gamma_1$ (cf. \eqref{eq:def-gamma}), defined by
\begin{equation}\label{eq:def-gamma-*}
\gamma^*_1:=\dfrac{(C_{a} + 1 + \beta^*_2 +\gamma_2)^2}{(\beta^*_2)^2\, \gamma_2},
\end{equation}
with $\gamma_2$ defined in \eqref{eq:def-gamma}. Then, given $\ubw_h\in\bW_h$ (cf. \eqref{eq:set-W-h}) there exists a unique $\ubd_h\in\bW_h$ such that $\cJ_h(\ubw_h) = \ubd_h$.
\end{lemma}
\begin{proof}
Given $\ubw_h=(\be_h,q_h)\in\bW_h$, we proceed analogously to the proof of Lemma~\ref{lem:well-def-J} and utilise \eqref{eq:bounded-b1-b2}, \eqref{eq:bounded-a}, \eqref{eq:infsup-b2-h}, \eqref{eq:infsup-b1-h}, \eqref{eq:coer-a-h} and \cite[Proposition 2.36]{ernguermond} to deduce the discrete inf-sup condition
\begin{equation}\label{eq:J-from-W-to-W-h}
\begin{array}{ll}
\|(\ubd_h,\bsigma_h),\bu_h\| \\[2ex]
\qquad \ds \leq \gamma^*_1\, \sup_{\bzero\neq((\ube_h,\btau_h),\bv_h)\in 
 (\bX_h\times\bY_h)\times\bZ_h } \frac{A_{\ubw_h}( (\ubd_h,\bsigma_h),(\ube_h,\btau_h) ) + b_2((\ubd_h,\bsigma_h),\bv_h) + b_2((\ube_h,\btau_h),\bu_h)}{\|(\ube_h,\btau_h),\bv_h\|}.
\end{array}
\end{equation}
Therefore, owing to the fact that for finite dimensional linear problems  surjectivity and injectivity are equivalent, from \eqref{eq:J-from-W-to-W-h} and the Banach--Ne\v cas--Babu\v ska theorem we obtain that there exists a unique $((\ubd_h,\bsigma_h),\bu_h)\in \bX_h\times\bY_h\times\bZ_h$ satisfying \eqref{eq:weak1-h-linearized}, with $\ubd_h\in\bW_h$, which concludes the proof.
\end{proof}

Using the same steps we can prove the following result. 
\begin{lemma}\label{lem:well-def-hat-J-h}
Given $r>0$, assume that 
\begin{equation}\label{eq:assumption-hat-J-h}
\dfrac{\tilde{\gamma}^*_1}{r} \big( 
 \|g\|_{0,\Omega}+ \|r_\Gamma \|_{-1/2,\Gamma} + \|\bu_\Gamma\|_{1/2,\Gamma} + \|\fb\|_{0,\Omega}   \big)\leq 1,
\end{equation}
where $\tilde{\gamma}^*_1$ is the discrete counterpart of $\tilde\gamma_1$ (cf. \eqref{eq:def-hat-gamma}), defined by
\begin{equation}\label{eq:def-hat-gamma-*}
\tilde{\gamma}^*_1:=\dfrac{(C_{a} + 1 + \tilde{\beta}^*_2 +\gamma_2)^2}{(\tilde{\beta}^*_2)^2\, \gamma_2},
\end{equation}
and with $\gamma_2$ defined in \eqref{eq:def-gamma}. Then, given $\ubw_h\in\widetilde{\bW}_h$ (cf. \eqref{eq:set-W-h}) there exists a unique $\ubd_h\in\widetilde{\bW}_h$ such that $\widetilde{\cJ}_h(\ubw_h) = \ubd_h$.
\end{lemma}

The following theorems provide the main result of this section, namely, existence and uniqueness of solution to the fixed-point
problems \eqref{eq:fixed-point-problem-1-h} and \eqref{eq:fixed-point-problem-2-h}, or equivalently, the wellposedness of problems \eqref{eq:weak1-h} and \eqref{eq:weak2-h}.

\begin{theorem}\label{theorem:unique-solution-weak1-h}
Let $\fb \in \bL^2(\Omega)$, $g \in L^2(\Omega)$, $\bu_\Gamma \in \bH^{1/2}(\Gamma)$ and $r_\Gamma \in \rH^{-1/2}(\Gamma)$ such that
\begin{equation}\label{eq:assumption-J-2-h}
\ds \dfrac{\gamma^*_1}{r}\max\{\gamma^*_1\,\kappa_2\,r,1\}\, \big( 
 \|g\|_{0,\Omega}+ \|r_\Gamma \|_{-1/2,\Gamma} + \|\bu_\Gamma\|_{1/2,\Gamma} + \|\fb\|_{0,\Omega}   \big) < 1,
\end{equation}
where $\gamma^*_1$ is defined in \eqref{eq:def-gamma-*}.
Then, the operator $\cJ_h$ (cf. \eqref{def:operator-J-h}) has a unique fixed point $\ubd_h\in\bW_h$. Equivalently, the problem \eqref{eq:weak1-h} has a unique solution $((\ubd_h,\bsigma_h),\bu_h)\in (\bX_h\times\bY_h)\times\bZ_h$ with $\ubd_h\in\bW_h$. In addition, there exists $C^*>0$ such that
\begin{equation}\label{eq:stability-weak1-h}
\|(\ubd_h,\bsigma_h),\bu_h\| \leq C^* \big( 
 \|g\|_{0,\Omega}+ \|r_\Gamma \|_{-1/2,\Gamma} + \|\bu_\Gamma\|_{1/2,\Gamma} + \|\fb\|_{0,\Omega}   \big).
\end{equation}
\end{theorem}
\begin{proof}
First we observe that, as for the continuous case, assumption \eqref{eq:assumption-J-2-h} ensures the well-definiteness of the operator $\cJ_h$. Now, adapting the arguments utilised in Theorem \ref{theorem:unique-solution-weak1}
one can obtain the following estimate
\begin{equation*}
\begin{array}{ll}
\|\cJ(\ubw_{1})- \cJ(\ubw_{2})\|_\bX \,=\, \|\ubd_{1} - \ubd_{2}\|_\bX \\[2ex]
 \qquad\ds\leq (\gamma^*_1)^2\,\kappa_2\,  \big( 
 \|g\|_{0,\Omega}+ \|r_\Gamma \|_{-1/2,\Gamma} + \|\bu_\Gamma\|_{1/2,\Gamma} + \|\fb\|_{0,\Omega}   \big) \|\ubw_1 - \ubw_2\|_{\bX}.
\end{array}
\end{equation*}
for all $\ubw_1, \ubw_2\in \bW_h$. In this way, using estimate \eqref{eq:assumption-J-2-h} we obtain that $\cJ_h$ is a contraction mapping on $\bW_h$, thus problem \eqref{eq:fixed-point-problem-1-h}, or equivalently
\eqref{eq:weak1-h} is wellposed. Finally, estimate \eqref{eq:stability-weak1-h} is obtained analogously to \eqref{eq:J-from-W-to-W}, which completes the proof.
\end{proof}

Using analogous arguments as above we can assert the following result.  
\begin{theorem}\label{theorem:unique-solution-weak2-h}
Let $\fb \in \bL^2(\Omega)$, $g \in L^2(\Omega)$, $\bu_\Gamma \in \bH^{1/2}(\Gamma)$ and $r_\Gamma \in \rH^{-1/2}(\Gamma)$ such that
\begin{equation}\label{eq:assumption-hat-J-2-h}
\ds \dfrac{\tilde{\gamma}^*_1}{r}\max\{\tilde{\gamma}^*_1\,\kappa_2\,r,1\}\, \big( 
 \|g\|_{0,\Omega}+ \|r_\Gamma \|_{-1/2,\Gamma} + \|\bu_\Gamma\|_{1/2,\Gamma} + \|\fb\|_{0,\Omega}   \big) < 1,
\end{equation}
where $\tilde{\gamma}^*_1$ is defined in \eqref{eq:def-hat-gamma-*}.
Then, the operator $\widetilde{\cJ}_h$ (cf. \eqref{def:operator-J-h}) has a unique fixed point $\ubd_h\in\widetilde{\bW}_h$. Equivalently,  problem \eqref{eq:weak2-h} has a unique solution $((\ubd_h,\bsigma_h),\ubu_h)\in \widetilde{\bX}_h\times\widetilde{\bY}_h\times\widetilde{\bZ}_h$ with $\ubd_h\in\widetilde{\bW}_h$. In addition, there exists $\tilde{C}^*>0$ such that
\begin{equation}\label{eq:stability-weak2-h}
\|(\ubd_h,\bsigma_h),\ubu_h\| \leq \tilde{C}^* \big( 
 \|g\|_{0,\Omega}+ \|r_\Gamma \|_{-1/2,\Gamma} + \|\bu_\Gamma\|_{1/2,\Gamma} + \|\fb\|_{0,\Omega}   \big).
\end{equation}
\end{theorem}


\section{A priori error estimates}\label{sec:error}
In this section, we aim to provide the convergence of the Galerkin schemes \eqref{eq:weak1-h} and \eqref{eq:weak1-h} and derive the corresponding rate of convergence. 

\subsection{Preliminaries}\label{sec:preliminaries}
From now on we assume that the hypotheses of Theorems \ref{theorem:unique-solution-weak1},\ref{theorem:unique-solution-weak1-h} and Theorems \ref{theorem:unique-solution-weak2},\ref{theorem:unique-solution-weak2-h} hold and let $((\ubd,\bsigma),\bu)\in (\bX\times\bY)\times\bZ$ and $((\ubd_h,\bsigma_h),\bu_h)\in (\bX_h\times\bY_h)\times\bZ_h$ be the unique solutions of \eqref{eq:weak1} and \eqref{eq:weak1-h}, respectively. In addition, let $((\ubd,\bsigma),\ubu)\in \widetilde{\bX}\times\widetilde{\bY}\times\widetilde{\bZ}$ and $((\ubd_h,\bsigma_h),\ubu_h)\in (\widetilde{\bX}_h\times\widetilde{\bY}_h)\times\widetilde{\bZ}_h$ be the unique solutions of \eqref{eq:weak2} and \eqref{eq:weak2-h}, respectively.

Then, similarly to \cite{CGO2021}, in order to simplify the subsequent analysis, we write $\texttt{e}_{\ubd} = \ubd - \ubd_h$, $\texttt{e}_{\bsigma} = \bsigma - \bsigma_h$ and $\texttt{e}_{\bu} = \bu - \bu_h$.
As usual, for a given $(\wh{\ube}_h,\wh{\bsigma}_h)\in\bX_{h} \times \bY_h$ and $\wh{\bv}_h\in \bZ_{h}$, we shall then decompose these errors into
\begin{equation}\label{eq:decompositions}
\texttt{e}_{\ubd} = \bxi_{\ubd} + \bchi_{\ubd},\qquad
\texttt{e}_{\bsigma} = \bxi_{\bsigma} + \bchi_{\bsigma},\qquad
\texttt{e}_{\bu} = \bxi_{\bu} + \bchi_{\bu},
\end{equation}
with
$$
\bxi_{\ubd} := \ubd - \wh{\ube}_h,\quad 
\bchi_{\ubd} := \wh{\ube}_h - \ubd_h,\quad 
\bxi_{\bsigma} := \bsigma - \wh{\btau}_h,\quad 
\bchi_{\bsigma} := \wh{\btau}_h - \bsigma_h,\quad 
\bxi_{\bu} := \bu - \wh{\bv}_h,\quad 
\bchi_{\bu} := \wh{\bv}_h - \bu_h.
$$
Recalling the definition of the bilinear form $A_{\ubw}$ in \eqref{eq:A}, from \eqref{eq:weak1} and \eqref{eq:weak1-h} we have that the following identities hold
\begin{equation*}
A_{\ubd}( (\ubd,\bsigma),(\ube,\btau) ) + b_2((\ube,\btau),\bu) + b_2((\ubd,\bsigma),\bv) = G(\ube) + H(\btau) + F(\bv),
\end{equation*}
for all $((\ube,\btau),\bv)\in (\bX\times\bY)\times\bZ$, and
\begin{equation*}
A_{\ubd_h}( (\ubd_h,\bsigma_h),(\ube_h,\btau_h) ) + b_2((\ube_h,\btau_h),\bu_h) + b_2((\ubd_h,\bsigma_h),\bv_h) = G(\ube_h) + H(\btau_h) + F(\bv_h).
\end{equation*}
for all $((\ube_h,\btau_h),\bv_h)\in (\bX_h\times\bY_h)\times\bZ_h$. From these relations, and similarly to \eqref{eq:A1-A2}, we can obtain that for all $((\ube_h,\btau_h),\bv_h)\in (\bX_h\times\bY_h)\times\bZ_h$, there holds
\begin{equation*}
A_{\ubd_h}( (\texttt{e}_{\ubd},\texttt{e}_{\bsigma}),(\ube_h,\btau_h) ) + b_2((\ube_h,\btau_h),\texttt{e}_{\bu}) + b_2((\texttt{e}_{\ubd},\texttt{e}_{\bsigma}),\bv_h) = \int_\Omega (\kappa(\ubd_h) - \kappa(\ubd) )\nabla p\cdot \nabla q_h,
\end{equation*}
which together with the definition of the errors in \eqref{eq:decompositions}, implies that 
\begin{equation}\label{eq:auxiliar-equation}
\begin{array}{ll}
A_{\ubd_h}( (\bchi_{\ubd},\bchi_{\bsigma}),(\ube_h,\btau_h) ) + b_2((\ube_h,\btau_h),\bchi_{\bu}) + b_2((\bchi_{\ubd},\bchi_{\bsigma}),\bv_h)\\[2ex]
\ds\quad= -A_{\ubd_h}( (\bxi_{\ubd},\bxi_{\bsigma}),(\ube_h,\btau_h) ) - b_2((\ube_h,\btau_h),\bxi_{\bu}) - b_2((\bxi_{\ubd},\bxi_{\bsigma}),\bv_h)+\int_\Omega (\kappa(\ubd_h) - \kappa(\ubd) ) \nabla p\cdot \nabla q_h,
\end{array}
\end{equation}
for all $((\ube_h,\btau_h),\bv_h)\in (\bX_h\times\bY_h)\times\bZ_h$. Then, since $\ubd_h \in \bW_h$, we apply the discrete inf-sup condition \eqref{eq:J-from-W-to-W-h} at the left-hand side of \eqref{eq:auxiliar-equation} followed by the continuity properties of $a$, $b_1$ and $b_2$ (cf. 
\eqref{eq:bounded-b1-b2} and \eqref{eq:bounded-a}) on the right-hand side of \eqref{eq:auxiliar-equation}, to obtain
\begin{equation}\label{eq:auxiliar-equation-2}
\begin{array}{ll}
\|\bchi_{\ubd}\|_\bX + \|\bchi_{\bsigma}\|_\bY + \|\bchi_{\bu}\|_\bZ \\[2ex]
\qquad \ds \leq \gamma^*_1\,\Big( (C_a + 2) \big(\|\bxi_{\ubd}\|_{\bX} + \|\bxi_{\bsigma}\|_{\bY} + \|\bxi_{\bu}\|_{\bZ} \big) + \kappa_2 \|p_h - p\|_{1,\Omega} \|\nabla p\|_{0,\Omega} \Big)\\[2ex]
\qquad \ds \leq \gamma^*_1\, \Big( (C_a + 2) \big(\|\bxi_{\ubd}\|_{\bX} + \|\bxi_{\bsigma}\|_{\bY} + \|\bxi_{\bu}\|_{\bZ} \big) + \kappa_2 (\|\bxi_{\ubd}\|_\bX + \|\bchi_{\ubd}\|_\bX) \|\ubd\|_{\bX} \Big).
\end{array}
\end{equation}
\subsection{Derivation of C\'ea estimates}
Now we turn to providing a best approximation estimate corresponding with the Galerkin scheme \eqref{eq:weak1-h}.
\begin{theorem}\label{th:cea-estimate}
Assume that
\begin{equation}\label{eq:assumption-cea}
\ds \kappa_2\gamma_1^* \gamma_1\, \big( 
 \|g\|_{0,\Omega}+ \|r_\Gamma \|_{-1/2,\Gamma} + \|\bu_\Gamma\|_{1/2,\Gamma} + \|\fb\|_{0,\Omega}   \big) \leq \dfrac{1}{2},
\end{equation}
with $\gamma_1$ and $\gamma^*_1$ being the constants in \eqref{eq:def-gamma} and \eqref{eq:def-gamma-*}. Furthermore, assume that the hypotheses of Theorem~\ref{theorem:unique-solution-weak1} and Theorem~\ref{theorem:unique-solution-weak1-h} hold. Let $((\ubd,\bsigma),\bu)\in (\bX\times\bY)\times\bZ$ and $((\ubd_h,\bsigma_h),\bu_h)\in (\bX_h\times\bY_h)\times\bZ_h$ be the unique solutions of \eqref{eq:weak1} and \eqref{eq:weak1-h}, respectively.
Then, there exists $C_{\text{C\'ea}}>0$, such that
\begin{equation}\label{eq:cea-estimate}
\ds \|((\ubd,\bsigma),\bu)-((\ubd_h,\bsigma_h),\bu_h)\| \leq C_{\text{C\'ea}} \inf_{((\ube_h,\btau_h),\bv_h)\in 
 (\bX_h\times\bY_h)\times\bZ_h} \|((\ubd,\bsigma),\bu)-((\ube_h,\btau_h),\bv_h)\|.
\end{equation}
\end{theorem}
\begin{proof}
From \eqref{eq:auxiliar-equation-2} we have 
\begin{equation}\label{eq:auxiliar-equation-4}
\begin{array}{ll}
\|\bchi_{\ubd}\|_\bX \big(1-\kappa_2\, \gamma^*_1 \, \|\ubd\|_{\bX} \big) + \|\bchi_{\bsigma}\|_\bY + \|\bchi_{\bu}\|_\bZ \\[2ex]
\qquad \ds \leq \gamma^*_1\, \Big( (C_a + 2) \big(\|\bxi_{\ubd}\|_{\bX} + \|\bxi_{\bsigma}\|_{\bY} + \|\bxi_{\bu}\|_{\bZ} \big) + \kappa_2 \|\bxi_{\ubd}\|_\bX \|\ubd\|_{\bX} \Big).
\end{array}
\end{equation}
Hence, using the fact that $\ubd$ satisfies \eqref{eq:J-from-W-to-W}, from assumption \eqref{eq:assumption-cea} and the latter inequality, we obtain
\begin{equation}\label{eq:auxiliar-equation-3}
\|\bchi_{\ubd}\|_\bX + \|\bchi_{\bsigma}\|_\bY + \|\bchi_{\bu}\|_\bZ  \leq C\, (\,\|\bxi_{\ubd}\|_{\bX} + \|\bxi_{\bsigma}\|_{\bY} + \|\bxi_{\bsigma}\|_{\bZ} \,),
\end{equation}
with $C>0$ independent of $h$. In this way, from \eqref{eq:decompositions}, \eqref{eq:auxiliar-equation-3} and the triangle inequality we obtain
\begin{equation*}
\|((\texttt{e}_{\ubd},\texttt{e}_{\bsigma}),\texttt{e}_{\bu})\|\leq \|((\bchi_{\ubd},\bchi_{\bsigma}),\bchi_{\bu})\| + \|((\bxi_{\ubd},\bxi_{\bsigma}),\bxi_{\bu})\| \leq (C+1) \|((\bxi_{\ubd},\bxi_{\bsigma}),\bxi_{\bu})\|,
\end{equation*}
which combined with the fact that $(\wh{\ube}_h,\wh{\bsigma}_h)\in\bX_{h} \times \bY_h$ and $\wh{\bv}_h\in \bZ_{h}$ are arbitrary, concludes the proof.
\end{proof}

\begin{remark}\label{rem:strang}
We note that, alternatively to the development in Theorem~\ref{th:cea-estimate} above, we could proceed as in the proof of \cite[Lemma 2.25]{ernguermond} and apply a Strang-type argument to obtain C\'ea’s estimate.
\end{remark}

Employing the same arguments utilised in Section~\ref{sec:preliminaries} and Theorem~\ref{th:cea-estimate}, we can provide the  C\'ea estimate corresponding with the Galerkin scheme \eqref{eq:weak2-h}.
\begin{theorem}\label{th:cea-estimate-2}
Assume that
\begin{equation}\label{eq:assumption-cea-2}
\ds \kappa_2\tilde{\gamma}_1^* \tilde{\gamma}_1\, \big( 
 \|g\|_{0,\Omega}+ \|r_\Gamma \|_{-1/2,\Gamma} + \|\bu_\Gamma\|_{1/2,\Gamma} + \|\fb\|_{0,\Omega}   \big) \leq \dfrac{1}{2},
\end{equation}
with $\tilde{\gamma}_1$ and $\tilde{\gamma}^*_1$ being the constants in \eqref{eq:def-hat-gamma} and \eqref{eq:def-hat-gamma-*}. Assume that the hypotheses of Theorem~\ref{theorem:unique-solution-weak2} and Theorem~\ref{theorem:unique-solution-weak2-h} hold. Let $((\ubd,\bsigma),\ubu)\in (\widetilde{\bX}\times\widetilde{\bY})\times\widetilde{\bZ}$ and $((\ubd_h,\bsigma_h),\ubu_h)\in (\widetilde{\bX}_h\times\widetilde{\bY}_h)\times\widetilde{\bZ}_h$ be the unique solutions of \eqref{eq:weak2} and \eqref{eq:weak2-h}, respectively.
Then, there exists $\tilde{C}_{\text{C\'ea}}>0$, such that
\begin{equation}\label{eq:cea-estimate-2}
\ds \|((\ubd,\bsigma),\bu)-((\ubd_h,\bsigma_h),\bu_h)\| \leq \tilde{C}_{\text{C\'ea}} \inf_{((\ube_h,\btau_h),\ubv_h)\in 
 (\widetilde{\bX}_h\times\widetilde{\bY}_h)\times\widetilde{\bZ}_h} \|((\ubd,\bsigma),\ubu)-((\ube_h,\btau_h),\ubv_h)\|.
\end{equation}
\end{theorem}
\begin{remark}
Using the estimates \eqref{eq:auxiliar-equation-4} and \eqref{eq:J-from-W-to-W}  together with the assumption \eqref{eq:assumption-cea}, we can obtain the following estimate for the C\'ea's constant $C_{\text{C\'ea}} = 2(\gamma^*_1 (C_a + 2) +1)$. Analogously, for the case of weakly imposed symmetry we can arrive at $\tilde{C}_{\text{C\'ea}} = 2(\tilde{\gamma}^*_1 (C_a + 2) +1)$. 
\end{remark}


\subsection{Rates of convergence}
In order to establish the rate of convergence of the Galerkin schemes \eqref{eq:weak1-h} and \eqref{eq:weak2-h}, we first recall the following approximation properties $\mathbf{AP}$ and $\widetilde{\mathbf{AP}}$ (interpolation estimates of Sobolev spaces for the twofold saddle-points for the poroelastic stress symmetry imposed strongly and weakly, respectively) associated with the finite element spaces  specified in Section~\ref{sec:spaces}.
\medskip

\begin{subequations}
\noindent $(\mathbf{AP}^{\bd}_h)$ For each $1\leq m \leq k+2$ and for each $\be\in \bbH^{m}(\Omega)\cap\bbLsym$, there holds
\begin{equation}\label{eq:rate-d}
\dist\big(\be,\bX_{1,h}\big) := \inf_{\be_h \in \bX_{1,h}} \|\be - \be_h\|_{0,\Omega} \lesssim h^{m}\,\|\be\|_{m,\Omega}.
\end{equation}

\noindent $(\mathbf{AP}^{p}_h)$ For each $0\leq m \leq k+1$ and for each $q\in \rH^{m+1}(\Omega)$, there holds
\begin{equation}\label{eq:rate-p}
\dist\big(q,\bX_{2,h}\big) := \inf_{q_h \in \bX_{2,h}} \|q - q_h\|_{1,\Omega} \lesssim h^{m}\,\|q\|_{m+1,\Omega}.
\end{equation}

\noindent $(\mathbf{AP}^{\bsigma}_h)$ For each $1\leq m \leq k+1$ and for each $\btau\in \bbH^{m}(\Omega)\cap \bbHsym$ with $\bdiv\,\btau\in \bH^{m}(\Omega)$, there holds
\begin{equation}\label{eq:rate-sigma}
\dist\big(\btau,\bY_h\big) := \inf_{\btau_h \in \bY_h} \|\btau - \btau_h\|_{\bdiv,\Omega} \lesssim h^{m}\,\Big\{ \|\btau\|_{m,\Omega} + \|\bdiv\,\btau\|_{m,\Omega} \Big\}.
\end{equation}

\noindent $(\mathbf{AP}^{\bu}_h)$ For each $1\leq m \leq k+1$ and for each $\bv\in \bH^{m+1}(\Omega)$, there holds
\begin{equation}\label{eq:rate-u}
\dist\big(\bv,\bZ_h\big) := \inf_{\bv_h \in \bZ_h} \|\bv - \bv_h\|_{0,\Omega} \lesssim  h^{m}\,\|\bv\|_{m+1,\Omega}.
\end{equation}
\end{subequations}

\noindent For \eqref{eq:rate-d}, \eqref{eq:rate-sigma} and \eqref{eq:rate-u} we refer to \cite[Theorem 6.1]{arnold-winther-0}, whereas  \eqref{eq:rate-p} can be found in \cite[Corollary 1.128]{ernguermond}.
On the other hand, regarding the formulations with weakly imposed stress symmetry, we have the following approximation bounds: 

\begin{subequations}
\noindent $(\widetilde{\mathbf{AP}}^{\bd}_h)$ For each $0\leq m \leq k+1$ and for each $\be\in \bbH^{m}(\Omega)$, there holds
\begin{equation}\label{eq:rate-d-2}
\dist\big(\be,\widetilde{\bX}_{1,h}\big) := \inf_{\be_h \in \widetilde{\bX}_{1,h}} \|\be - \be_h\|_{0,\Omega} \lesssim h^{m}\,\|\be\|_{m,\Omega}.
\end{equation}

\noindent $(\widetilde{\mathbf{AP}}^{p}_h)$ Coincides with  $({\mathbf{AP}}^{p}_h)$.

\noindent $(\widetilde{\mathbf{AP}}^{\bsigma}_h)$ For each $0 < m \leq k+1$ and for each $\btau\in \bbH^{m}(\Omega)\cap \bbH(\bdiv;\Omega)$ with $\bdiv\,\btau\in \bH^{m}(\Omega)$, there holds
\begin{equation}\label{eq:rate-sigma-2}
\dist\big(\btau,\widetilde{\bY}_h\big) := \inf_{\btau_h \in \widetilde{\bY}_h} \|\btau - \btau_h\|_{\widetilde{Y}} \lesssim h^{m}\,\Big\{ \|\btau\|_{m,\Omega} + \|\bdiv\,\btau\|_{m,\Omega} \Big\}.
\end{equation}

\noindent $(\widetilde{\mathbf{AP}}^{\bu}_h)$ For each $0\leq m \leq k+1$ and for each $\bv\in \bH^{m}(\Omega)$, there holds
\begin{equation}\label{eq:rate-u-2}
\dist\big(\bv,\widetilde{\bZ}_{1,h}\big) := \inf_{\bv_h \in \widetilde{\bZ}_{1,h}} \|\bv - \bv_h\|_{0,\Omega} \lesssim  h^{m}\,\|\bv\|_{m,\Omega}.
\end{equation}

\noindent $(\widetilde{\mathbf{AP}}^{\bgamma}_h)$ For each $0\leq m \leq k+1$ and for each $\bbeta\in \bbH^{m}(\Omega)\cap \bbL^2_{\rskew}(\Omega)$, there holds
\begin{equation}\label{eq:rate-gamma}
\dist\big(\bbeta,\widetilde{\bZ}_{2,h}\big) := \inf_{\bbeta_h \in \widetilde{\bZ}_{2,h}} \|\bbeta - \bbeta_h\|_{0,\Omega} \lesssim  h^{m}\,\|\bbeta\|_{m,\Omega}.
\end{equation}
\end{subequations}
For \eqref{eq:rate-d-2}, \eqref{eq:rate-sigma-2}, \eqref{eq:rate-u-2} and \eqref{eq:rate-gamma} we refer to \cite[Theorem 2.4]{gatica2013priori}.

With these steps we are now in position to state the rates of convergence associated with the Galerkin schemes \eqref{eq:weak1-h} and \eqref{eq:weak2-h}.

\begin{theorem}\label{th:rate1}
Assume that the hypotheses of Theorem~\ref{th:cea-estimate} hold and let
$((\ubd,\bsigma),\bu)\in (\bX\times\bY)\times\bZ$ and $((\ubd_h,\bsigma_h),\bu_h)\in (\bX_h\times\bY_h)\times\bZ_h$ be the unique solutions of the continuous and discrete problems \eqref{eq:weak1} and \eqref{eq:weak1-h}, respectively.
Assume further that $\bd\in \bbH^{m}(\Omega)$, $p\in \rH^{m+1}(\Omega)$, $\bsigma\in \bbH^{m}(\Omega)$, $\bdiv \,\bsigma \in \bH^{m}(\Omega)$ and $\bu \in \bH^{m+1}(\Omega)$, for $1\leq m \leq k+1$.
Then there exists $C_{rate}>0$, independent of $h$, such that 
\begin{equation*}
\ds\|((\texttt{e}_{\ubd},\texttt{e}_{\bsigma}),\texttt{e}_{\bu})\|
\,\leq\,
C_{rate}\,h^{m}\Big\{ \|\bd\|_{m,\Omega} \,+\, \|p\|_{m+1,\Omega}
\,+\,\|\bsigma\|_{m,\Omega} \,+\, \|\bdiv\,\bsigma\|_{m,\Omega} \,+\, \|\bu\|_{m+1,\Omega}\Big\}.
\end{equation*}
\end{theorem} 
\begin{proof}
The result is a straightforward application of Theorem~\ref{th:cea-estimate}
and the approximation properties $(\mathbf{AP}^{\bd}_h)$, $(\mathbf{AP}^{p}_h)$, $(\mathbf{AP}^{\bsigma}_h)$, and $(\mathbf{AP}^{\bu}_h)$.
\end{proof}	

\begin{theorem}\label{th:rate2}
Assume that the hypotheses of Theorem~\ref{th:cea-estimate-2} hold and let
$((\ubd,\bsigma),\ubu)\in \widetilde{\bX}\times\widetilde{\bY}\times\widetilde{\bZ}$ and $((\ubd_h,\bsigma_h),\ubu_h)\in (\widetilde{\bX}_h\times\widetilde{\bY}_h)\times\widetilde{\bZ}_h$ be the unique solutions of the continuous and discrete problems \eqref{eq:weak2} and \eqref{eq:weak2-h}, respectively.
Assume further that $\bd\in \bbH^{m}(\Omega)$, $p\in \rH^{m+1}(\Omega)$, $\bsigma\in \bbH^{m}(\Omega)$, $\bdiv\, \bsigma \in \bH^{m}(\Omega)$, $\bu \in \bH^{m}(\Omega)$ and $\bgamma\in \bbH^{m}(\Omega)$, for $0\leq m \leq k+1$.
Then there exists $\widetilde{C}_{rate}>0$, independent of $h$, such that 
\begin{equation*}
\ds\|((\texttt{e}_{\ubd},\texttt{e}_{\bsigma}),\texttt{e}_{\ubu})\|
\,\leq\,
\widetilde{C}_{rate}\,h^{m}\Big\{ \|\bd\|_{m,\Omega} \,+\, \|p\|_{m+1,\Omega}
\,+\,\|\bsigma\|_{m,\Omega} \,+\, \|\bdiv\,\bsigma\|_{m,\Omega} \,+\, \|\bu\|_{m,\Omega} \,+\,  \|\bgamma\|_{m,\Omega} \Big\}.
\end{equation*}
\end{theorem} 
\begin{proof}
The result is a straightforward application of Theorem~\ref{th:cea-estimate-2}
and the approximation properties $(\widetilde{\mathbf{AP}}^{\bd}_h)$, $(\widetilde{\mathbf{AP}}^{p}_h)$, $(\widetilde{\mathbf{AP}}^{\bsigma}_h)$, $(\widetilde{\mathbf{AP}}^{\bu}_h)$ and $(\widetilde{\mathbf{AP}}^{\bgamma}_h)$.
\end{proof}



\section{Numerical results}\label{sec:numer}
This section contains selected computational examples that serve to confirm the theoretically obtained convergence rates of the two analysed mixed finite element formulations. We showcase tests in 2D and 3D including experimental convergence as well as an application-oriented simulation of filtration of interstitial fluid in soft tissue. The implementation has been carried out using the finite element library \texttt{Firedrake} \cite{rathgeber2016firedrake}. Each solve of the discrete nonlinear coupled system was performed using Newton--Raphson's method with iterations terminated whenever the absolute or the relative $\ell^\infty-$norm of the discrete residual in the product space drops below the fixed tolerance $10^{-7}$. Each linear system arising from linearisation was solved using the sparse LU factorisation algorithm  MUMPS. 

\subsection{Verification of convergence with respect to smooth solutions}

\begin{table}[t!]
	\setlength{\tabcolsep}{4.5pt}
	\begin{center}  
\begin{tabular}{|cccccccccc|}
	\hline 	
	DoFs & $h$ & $\texttt{e}_0(\bd)$ & \texttt{rate} & $\texttt{e}_1({p})$ & \texttt{rate}  & $\texttt{e}_{\bdiv}(\bsigma)$ & \texttt{rate} & $\texttt{e}_0(\bu)$ & rate\\ 
	\hline
 		\hline
   303 & 0.5000 & 1.5e-03 & $\star$ & 3.8e-01 & $\star$ & 8.4e-02 & $\star$ & 6.3e-03 & $\star$ \\
  1063 & 0.2500 & 3.1e-04 & 2.24 & 1.2e-01 & 1.70 & 2.2e-02 & 1.92 & 1.6e-03 & 2.01 \\
  3975 & 0.1250 & 4.9e-05 & 2.68 & 3.2e-02 & 1.88 & 5.6e-03 & 1.98 & 3.9e-04 & 2.00 \\
 15367 & 0.0625 & 7.2e-06 & 2.76 & 8.2e-03 & 1.95 & 1.4e-03 & 2.00 & 9.8e-05 & 2.00 \\
 60423 & 0.0312 & 1.1e-06 & 2.72 & 2.1e-03 & 1.98 & 3.5e-04 & 2.00 & 2.5e-05 & 2.00 \\
239623 & 0.0156 & 2.3e-07 & 2.03 & 5.3e-04 & 1.98 & 8.7e-05 & 2.00 & 6.1e-06 & 2.00 \\
\hline
		\end{tabular}
		\end{center}
		
		\smallskip
		\caption{Verification of convergence for the method imposing stress symmetry strongly and using   AW$_{k}$-based elements (using $k=1$: this gives polynomial degree
  3 for stress and strain and 1  for displacement). Errors and convergence rates are tabulated for strain, fluid pressure, poroelastic stress, and displacement.  The symbol $\star$ in the first mesh refinement level indicates that no convergence rate is computed.}\label{table:1h}
	\end{table}

\begin{table}[t!]
	\setlength{\tabcolsep}{4.5pt}
	\begin{center}  
	\begin{tabular}{|cccccccccccc|}
	\hline 	
	DoFs & $h$ & $\texttt{e}_0(\bd)$ & \texttt{rate} & $\texttt{e}_1({p})$ & \texttt{rate}  & $\texttt{e}_{\bdiv}(\bsigma)$ & \texttt{rate} & $\texttt{e}_0(\bu)$ & \texttt{rate} & $\texttt{e}_0(\bgamma)$ & \texttt{rate} \\ 
	\hline
	\hline
\multicolumn{12}{|c|}{PEERS$_k$-based FE scheme with $k=0$}\\
				\hline
      130 & 0.7071 & 1.5e-01 & $\star$ & 1.1e+0 & $\star$ & 1.2e+0 & $\star$ & 4.6e-02 & $\star$ & 8.4e-02 & $\star$ \\
   482 & 0.3536 & 9.4e-02 & 0.66 & 7.4e-01 & 0.62 & 6.3e-01 & 0.90 & 2.3e-02 & 0.99 & 4.9e-02 & 0.79\\
  1858 & 0.1768 & 5.2e-02 & 0.86 & 4.1e-01 & 0.85 & 3.2e-01 & 0.98 & 1.1e-02 & 1.02 & 2.7e-02 & 0.83\\
  7298 & 0.0884 & 2.6e-02 & 0.97 & 2.1e-01 & 0.95 & 1.6e-01 & 1.00 & 5.6e-03 & 1.01 & 1.1e-02 & 1.26\\
 28930 & 0.0442 & 1.3e-02 & 1.00 & 1.1e-01 & 0.98 & 8.0e-02 & 1.00 & 2.8e-03 & 1.00 & 4.3e-03 & 1.42\\
115202 & 0.0221 & 6.6e-03 & 1.00 & 5.4e-02 & 1.00 & 4.0e-02 & 1.00 & 1.4e-03 & 1.00 & 1.5e-03 & 1.48\\
			\hline
\multicolumn{12}{|c|}{PEERS$_k$-based FE scheme with $k=1$}\\
				\hline	
       386 & 0.7071 & 2.3e-02 & $\star$ & 3.8e-01 & $\star$ & 3.4e-01 & $\star$ & 5.0e-03 & $\star$ & 9.5e-03 & $\star$\\
  1474 & 0.3536 & 8.1e-03 & 1.51 & 1.2e-01 & 1.70 & 9.0e-02 & 1.91 & 1.2e-03 & 2.02 & 3.1e-03 & 1.63\\
  5762 & 0.1768 & 2.6e-03 & 1.66 & 3.2e-02 & 1.88 & 2.3e-02 & 1.97 & 3.0e-04 & 2.02 & 1.4e-03 & 1.64\\
 22786 & 0.0884 & 7.3e-04 & 1.81 & 8.2e-03 & 1.95 & 5.8e-03 & 1.99 & 7.5e-05 & 2.02 & 4.9e-04 & 1.79\\
 90626 & 0.0442 & 1.9e-04 & 1.91 & 2.1e-03 & 1.98 & 1.5e-03 & 1.99 & 1.9e-05 & 2.01 & 1.4e-04 & 1.89 \\
361474 & 0.0221 & 5.0e-05 & 1.96 & 5.2e-04 & 1.99 & 3.7e-04 & 2.00 & 4.7e-06 & 2.00 & 3.8e-05 & 1.98\\
\hline
    \multicolumn{12}{|c|}{AFW$_k$-based FE scheme with $k=0$}\\
				\hline
       161 & 0.7071 & 2.5e-02 & $\star$ & 1.1e+0 & $\star$ & 1.1e+0 & $\star$ & 4.4e-02 & $\star$ & 2.6e-02 & $\star$ \\
   569 & 0.3536 & 1.5e-02 & 0.72 & 7.4e-01 & 0.62 & 5.7e-01 & 0.93 & 2.2e-02 & 0.98 & 1.4e-02 & 0.93\\
  2129 & 0.1768 & 5.9e-03 & 1.35 & 4.1e-01 & 0.85 & 2.9e-01 & 0.99 & 1.1e-02 & 1.00 & 6.3e-03 & 1.09\\
  8225 & 0.0884 & 2.4e-03 & 1.30 & 2.1e-01 & 0.95 & 1.5e-01 & 1.00 & 5.6e-03 & 1.00 & 3.1e-03 & 1.05\\
 32321 & 0.0442 & 1.1e-03 & 1.13 & 1.1e-01 & 0.98 & 7.3e-02 & 1.00 & 2.8e-03 & 1.00 & 1.5e-03 & 1.02\\
128129 & 0.0221 & 5.4e-04 & 1.04 & 5.4e-02 & 1.00 & 3.6e-02 & 1.00 & 1.4e-03 & 1.00 & 7.5e-04 & 1.00\\
			\hline
\multicolumn{12}{|c|}{AFW$_k$-based FE scheme with $k=1$}\\
				\hline	
      385 & 0.7071 & 5.3e-03 & $\star$ & 3.8e-01 & $\star$ & 3.3e-01 & $\star$ & 4.8e-03 & $\star$ & 3.0e-03 & $\star$ \\
  1425 & 0.3536 & 1.3e-03 & 2.05 & 1.2e-01 & 1.70 & 8.7e-02 & 1.92 & 1.2e-03 & 2.00 & 9.7e-04 & 1.65 \\
  5473 & 0.1768 & 2.1e-04 & 2.61 & 3.2e-02 & 1.88 & 2.2e-02 & 1.98 & 3.0e-04 & 2.00 & 1.9e-04 & 2.38\\
 21441 & 0.0884 & 3.4e-05 & 2.62 & 8.2e-03 & 1.95 & 5.5e-03 & 2.00 & 7.5e-05 & 2.00 & 3.8e-05 & 2.30\\
 84865 & 0.0442 & 6.5e-06 & 2.38 & 2.1e-03 & 1.98 & 1.4e-03 & 2.00 & 1.9e-05 & 2.00 & 8.7e-06 & 2.12\\
337665 & 0.0221 & 1.5e-06 & 2.15 & 5.2e-04 & 1.99 & 3.5e-04 & 2.00 & 4.7e-06 & 2.00 & 2.1e-06 & 2.04\\
\hline
		\end{tabular}
		\end{center}
		
		\smallskip
		\caption{Verification of convergence for the method imposing stress symmetry weakly and using PEERS$_{k}$ \eqref{fe:peers} and AFW$_{k}$ \eqref{fe:afw} elements with polynomial degrees $k=0$ and $k=1$. Errors and convergence rates are tabulated for strain, fluid pressure, poroelastic stress, displacement, and rotation Lagrange multiplier. The $\star$ in the first mesh refinement level indicates that no convergence rate is computed.}\label{table:2h}
	\end{table}

First we conduct a test of convergence where model parameters assume the following values  $\mu = \lambda =  \mu_f=1$ and $c_0 =\alpha = \frac{1}{4}$. The nonlinear permeability is taken as the second form in  \eqref{eq:kappa} (the Kozeny--Carman law) with $k_0 = k_1 = 0.1$. We use the following closed-form smooth solutions to the primal form of the coupled nonlinear problem 
\[
	\bu = \frac{1}{5}\begin{pmatrix}
	-x_1\cos(x_1)\sin(x_2)+{x_1^2}\\
	x_1\sin(x_1)\cos(x_2)+{x_2^2}
	\end{pmatrix}, \quad  
	p = \sin(\pi x_1)\sin(\pi x_2),
\]
which are used to generate exact mixed variables $\bd,\bsigma, \bgamma$, and to produce non-homogeneous forcing term $\fb$, boundary data $\bu_\Gamma,r_\Gamma$, and source term $g$. Seven successively refined meshes (congruent right-angled triangular partitions) are generated for the domain $\Omega =(0,1)^2$ and we compute  errors between approximate and exact solutions $e(\cdot)$  (measured in the $\rH^1$-norm for fluid pressure, in the tensor $\bbH(\bdiv)$-norm for poroelastic Cauchy stress, and in the tensorial and vectorial $L^2$-norms for strain, rotation, and displacement).  The experimental rates of convergence at each mesh refinement are computed as 
 \[\texttt{rate} (\%)=\frac{\log( \texttt{e}(\%)/ \texttt{e}'(\%))}{\log(h/h')},\]
 with \% $\in \{\bd,p,\bsigma,\bu,\bgamma\}$, and where $\texttt{e},\texttt{e}'$ stand for errors generated on two consecutive meshes of sizes $h,h'$. 
 The mixed finite element methods are defined by the 
 conforming AW$_{k}$, AFW$_{k}$, and PEERS$_{k}$-type of spaces specified in Section~\ref{sec:FE}.

Tables~\ref{table:1h}-\ref{table:2h} show the error history associated with the formulations with strong and weak symmetry imposition, and using for the latter case the two lowest-order polynomial degrees. In all runs  our results confirm an error decay  with a convergence rate of $O(h^{k+1})$ for all field variables in their natural norms, which is consistent with the theoretical error bounds from Theorems \ref{th:rate1}-\ref{th:rate2}. We also depict examples of approximate solutions computed with the PEERS$_k$-based finite element family (setting $k=1$). See Figure~\ref{fig:ex01}, where the panels show also the outline of the domain  before  the deformation. 

\begin{figure}[t!]
    \centering
    \includegraphics[width=0.3\textwidth]{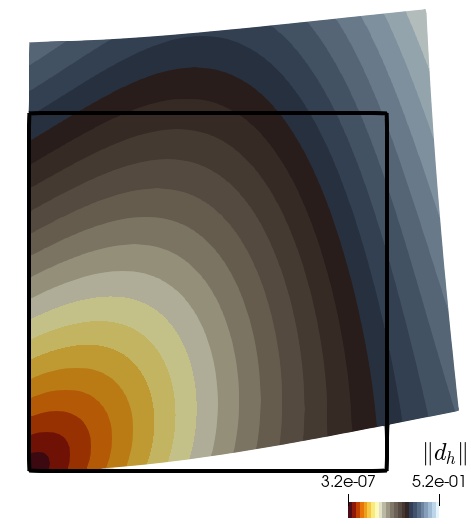}
    \includegraphics[width=0.3\textwidth]{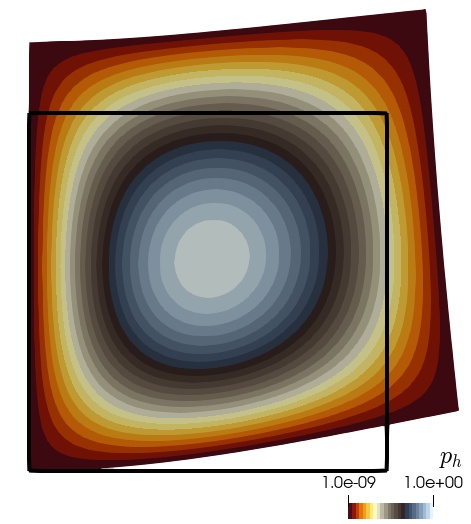}
    \includegraphics[width=0.3\textwidth]{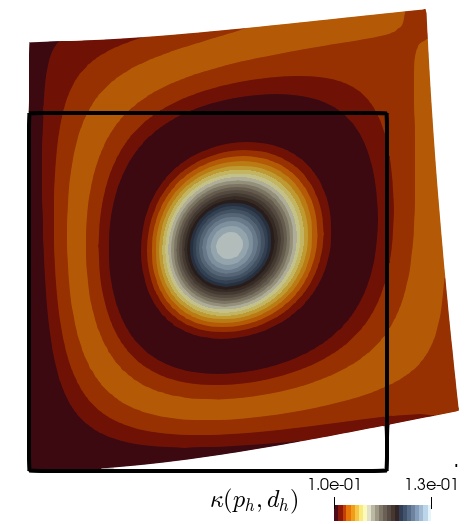}
    \includegraphics[width=0.3\textwidth]{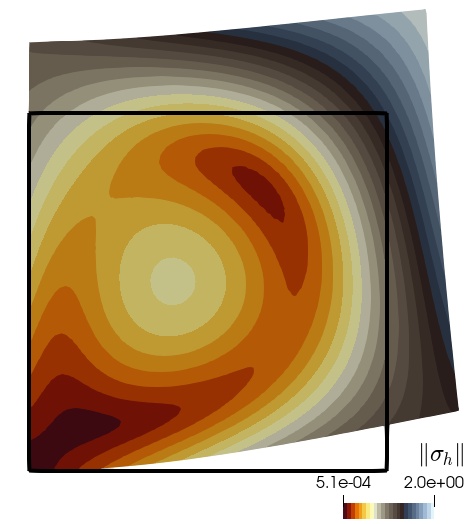}
    \includegraphics[width=0.3\textwidth]{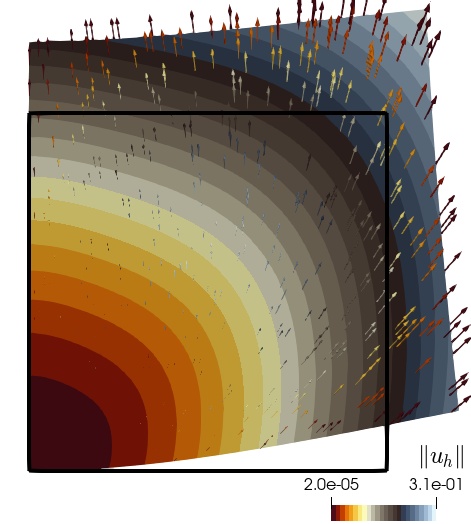}
    \includegraphics[width=0.3\textwidth]{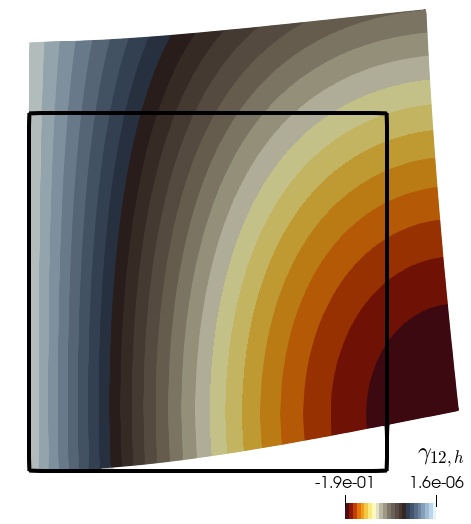}
    \caption{Convergence tests. Approximate solutions computed with the second-order PEERS$_k$-based finite element scheme and rendered on the deformed configuration. Only the non-trivial component of the rotation tensor is shown in the bottom-right panel.}
    \label{fig:ex01}
\end{figure}

\subsection{Simulation of swelling of a porous structure}\label{sec:swelling}
In the next test we replicate the swelling of a 3D block. The parameters and domain configuration are taken similarly to  \cite{oyarzua16} and we simulate this behaviour with the second-order AFW$_k$-based finite element method using \eqref{fe:afw}-\eqref{eq:hat-X-1-h-b}. The domain is $\Omega=(0,1)\times(0,1)\times(0,\frac{1}{2})$ and the deformation is induced by a fluid pressure gradient in the $x_1-$direction (we impose $p=10^4$ at $x_1=0$, $p=0$ at $x_1=1$, and zero-flux conditions  on the remainder of $\Gamma$). The poroelastic body is allowed to slide on the sides $x_1=0$, $x_2=0$ and $x_3=0$ (these parts are called $\Gamma_{\mathrm{slide}}$),  
whereas zero normal stress is considered elsewhere on $\Gamma$. The sliding condition $\bu\cdot\bn|_{\Gamma_{\mathrm{slide}}} = 0$ is incorporated through the additional term 
\[ \langle (\btau\bn)\times \bn, \bu\times \bn \rangle_{\Gamma_{\mathrm{slide}}},\]
on the left-hand side of the second equation in the weak formulation \eqref{eq:weak1} and taking $H(\btau) = 0$. 
The model parameters for this test are the exponential permeability in \eqref{eq:kappa} with $k_0=10^{-9}$, $k_1 = 10^{-6}$, $k_2 = -0.5$, the Young modulus 
$E=8000$, Poisson ratio $\nu=0.3$, storativity coefficient $c_0=0.001$,  Biot--Willis parameter $\alpha=0.5$, fluid viscosity  
$\mu_f = 10^{-3}$, and we take zero body loads and volumetric sources. Figure~\ref{fig:ex02} displays the approximate solutions rendered on the deformed configuration. We also show the contour of the undeformed domain for reference.

\subsection{Poroelastic filtration of slightly compressible trabecular meshwork}
For our next example we consider a computational domain extracted and meshed from imaging of trabecular meshwork tissue in the canine eye in \cite{ruiz22}.  The characteristic length of the domain is $6.8\cdot10^{-3}\,[\text{m}]$. In this test we use the strongly imposed symmetry with AW$_{k}$-based finite elements \eqref{fe:aw}-\eqref{eq:X-1-h}, and take piecewise linear and overall continuous elements for the fluid pressure. 
We set up an initial porosity field $\phi_0$, randomly distributed between 0.3 and 0.45. A nonlinear permeability is prescribed depending on that initial porosity and on fluid pressure and skeleton dilation
\[\kappa(\zeta) = \frac{k_0}{\mu_f} + \frac{k_1}{\mu_f}\exp(-\frac{1}{2}\zeta) \quad \,[\text{in m}^2], \quad \text{and}\quad \zeta = \phi_0+(1-\phi_0)(c_0p+\alpha\tr\bd), \]
where the form for $\zeta$ is similar as in \cite{van2023mathematical}. 
The model parameters for this test are $k_0=10^{-10}$, $k_1 = 10^{-7}$, the Lam\'e constants $\lambda = 14388\,[\text{Pa}]$, $\mu = 1102\,[\text{Pa}]$,  storativity coefficient $c_0=0.05$,  Biot--Willis parameter $\alpha=0.95$, fluid viscosity  
$\mu_f = 7.54^{-4}\, [\text{Pa}\cdot \text{s}]$, and we take zero body loads and volumetric sources. The domain is assumed in contact, on a portion of the boundary on the top-left end, with the anterior chamber in the eye and therefore we set a traction of $\sigma\bn = -3\times 10^{-3}\bn$ and a pore pressure $p=2\times 10^3\,[\text{Pa}]$. On the outlet sub-boundary (a small region on the bottom-right end) we impose zero fluid pressure and traction-free conditions, and on the remainder of the boundary we set zero displacements and zero flux for the fluid pressure. 
From the results portrayed in Figure~\ref{fig:ex03}  we observe that the pore pressure and strain concentration generation near the interfacial region  imply a smaller  permeability, which progressively increases as one approaches the outlet boundary. This behaviour coincides with the first round of tests with different  permeability profiles explored in \cite{ruiz22}. 
We also plot the off-diagonal entries of the Cauchy stress to illustrate the balance of angular momentum, and on the bottom-right panel we can see the deformation of the interfacial region and, as expected, a smaller expansion of the tissue towards the outlet.

\begin{figure}[!t]
    \centering
    \includegraphics[width=0.325\textwidth]{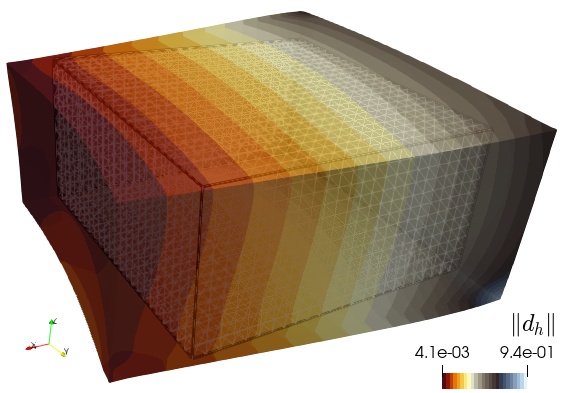}
    \includegraphics[width=0.325\textwidth]{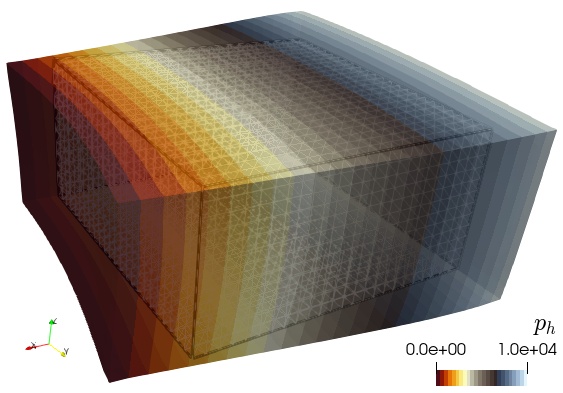}
    \includegraphics[width=0.325\textwidth]{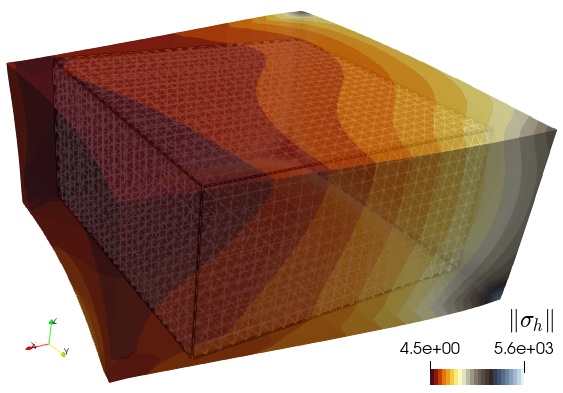}
    \includegraphics[width=0.325\textwidth]{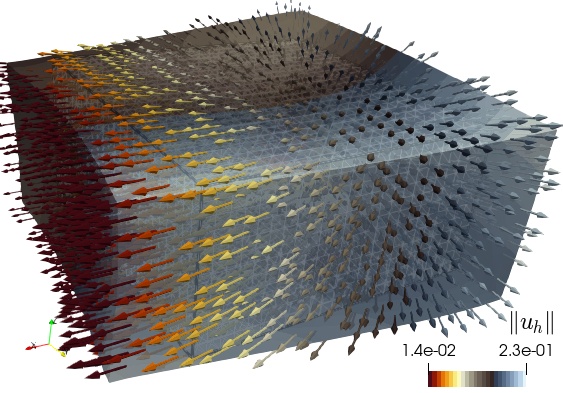}
    \includegraphics[width=0.325\textwidth]{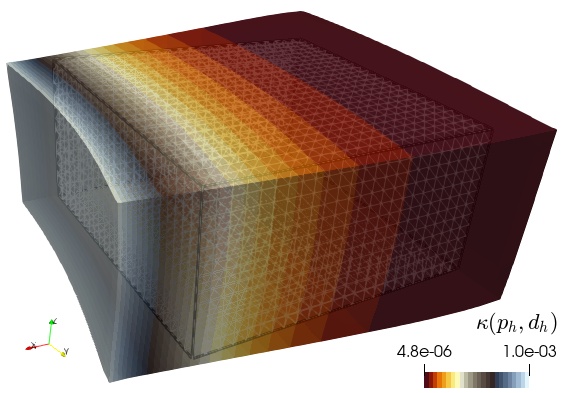}
    \caption{Swelling of a poroelastic block. Approximate infinitesimal strain magnitude, fluid pressure, poroelastic stress magnitude, displacement magnitude and direction arrows, and permeability distribution.}
    \label{fig:ex02}
\end{figure}
 
\subsection{Reproducing the Mandel effect}
To conclude this section, we utilise the proposed formulation, specifically focusing on the scenario of weakly symmetric stress, to simulate Mandel's effect (see, e.g., \cite{mandel1953consolidation} or also  \cite{brun2020monolithic,castelletto2015accuracy, coussy04,fjaer2021petroleum,guo2020multiple,teichtmeister2019aspects}). Such a problem involves a specimen made of isotropic poroelastic material, which is positioned between two rigid frictionless impervious plates at the top and bottom. The slab is infinitely long with a cross section measuring $2L\times 2H$. The lateral sides of the specimen are free and permeable. In this simulation, a compressive force is exerted on the horizontal plates. As a result, the pore pressure at the centre of the specimen surpasses its starting value during the early stages of the process and subsequently diminishes until it reaches zero. This behavior can be attributed to the drainage of fluid from the specimen through the side edges. Consequently, a greater portion of the applied load is transferred towards the comparatively stiffer central region of the specimen.

\begin{figure}[!t]
    \centering
    \includegraphics[width=0.325\textwidth]{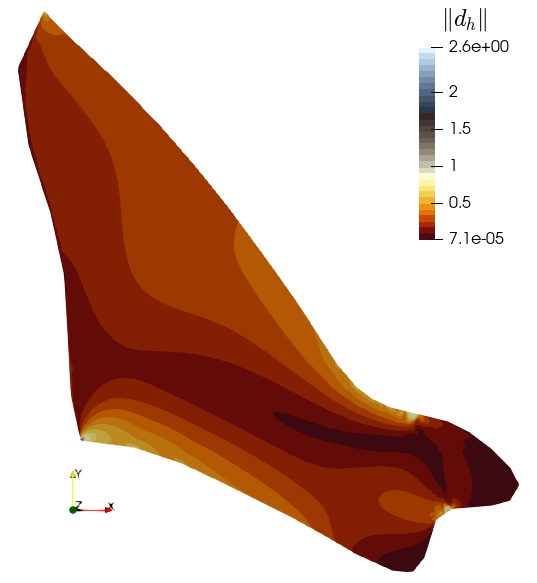}
    \includegraphics[width=0.325\textwidth]{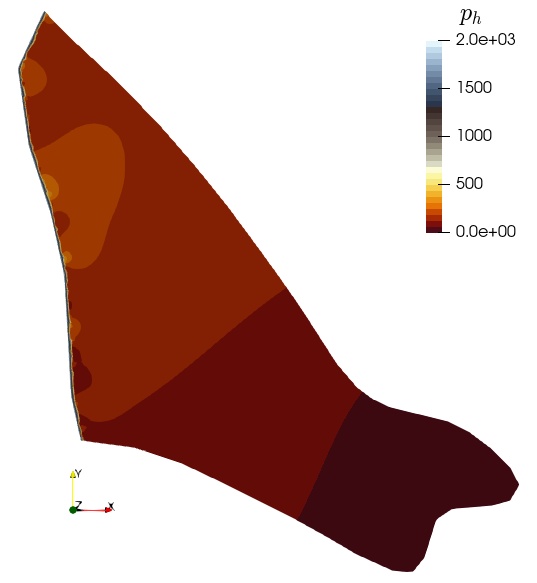}
    \includegraphics[width=0.325\textwidth]{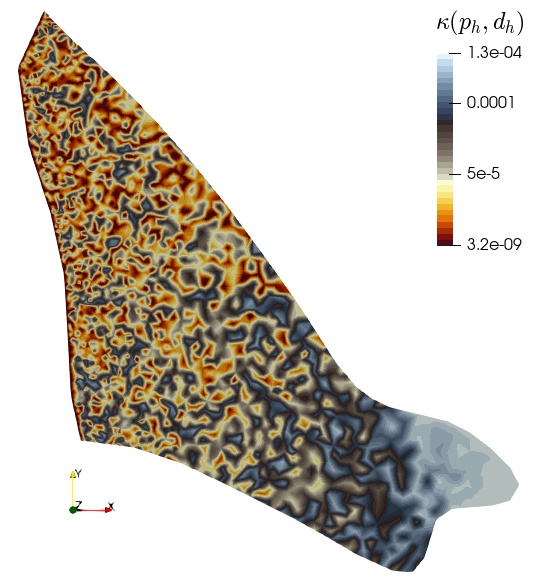}\\
    \includegraphics[width=0.325\textwidth]{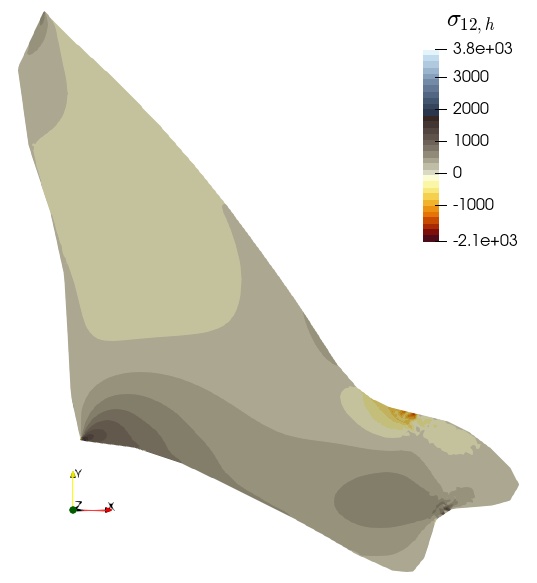}
    \includegraphics[width=0.325\textwidth]{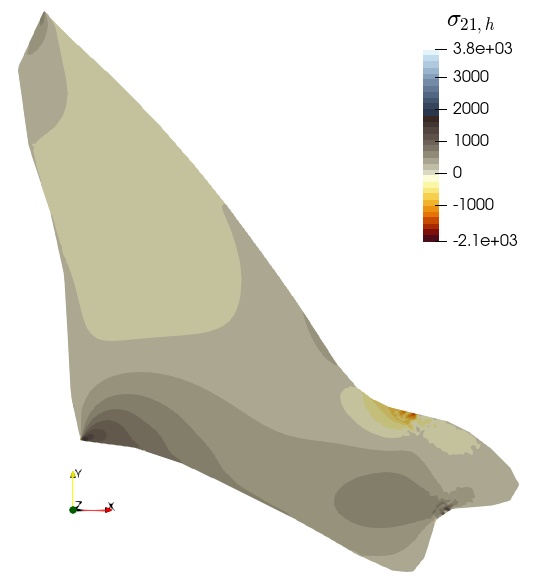}
    \includegraphics[width=0.325\textwidth]{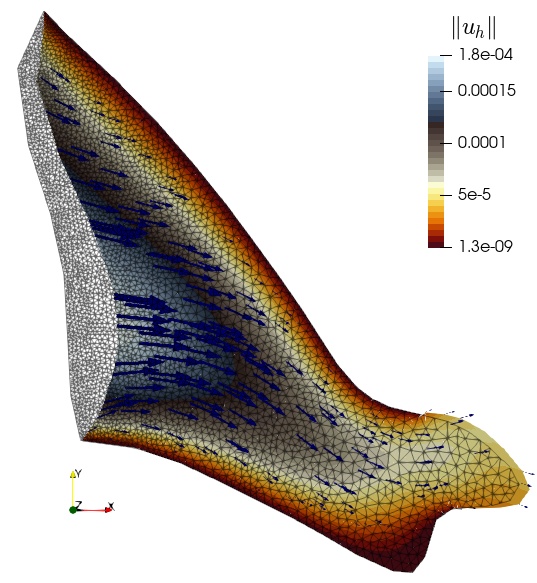}
    \caption{Poroelastic filtration of trabecular meshwork tissue. Approximate infinitesimal strain magnitude, fluid pressure, nonlinear permeability (depending on an initial randomly distributed porosity), two off-diagonal entries of the poroelastic stress, and displacement magnitude with direction arrows.}
    \label{fig:ex03}
\end{figure}

To simplify the analysis, as usual, taking into account the symmetry of the geometry and problem set up, we only consider a quarter of the entire domain: $\Omega = (0,L)\times(0,H)$. This implies that the mechanical boundary conditions for the smaller domain $\Omega$ are as follows: on the boundary $x=L$ the pore pressure is fixed at zero and the normal poroelastic stress is also zero (both imposed as essential boundary conditions). At the boundaries $x=0$ and $y=0$ we impose a sliding condition $\bu\cdot\bn=0$ for which, as in Section~\ref{sec:swelling}, the term $-\langle (\btau\bn) \cdot\bt, \bu \cdot \bt\rangle$ appears on the right-hand side of the weak form (where $\bt$ denotes the tangent vector to the boundary). A downward force of magnitude $2F$ is applied to the top plate (the boundary $y=H$) through the essential boundary condition for poroelastic stress $\bsigma\bn = (0,-F)^{\tt t}$. Zero-flux for the fluid phase is considered in all sub-boundaries except on $x=L$. This problem takes place in the quasi-steady regime and therefore we bring back in the time dependence, discretised using backward Euler's method with constant time step $\Delta t$ and running until $T_{\mathrm{end}}$, and necessitating an initial condition for pore pressure and strain (we initialise them both to zero). A coarse mesh is used together with the AFW$_k$ finite element family with $k=1$,  and following e.g. \cite{li1999nonlinear} (where a thorough computational and experimental comparison is performed for poroelastic cartilage tissue in different regimes), we test the behaviour of the models with constant and nonlinear permeabilities $\kappa=\kappa_0$ and $\kappa = k_0\kappa_0\exp(k_1\zeta)$. 

We set the geometry and model parameters within the ranges used in \cite{fjaer2021petroleum}
\begin{gather*}
L=H=1\,[\mathrm{m}], \quad T_{\mathrm{end}} = 1\,[\mathrm{s}], \quad \Delta t = 0.01\,[\mathrm{s}], \quad 
c_0 = 4\times 10^{-10}\,[\mathrm{Pa}^{-1}], \quad \kappa_0 = 5.1\times10^{-8}\,[\mathrm{m}^2], \\
k_0 = 5\,[-], \quad k_1 = 30\,[-], \quad \alpha = 0.9\,[-], \quad \mu_f = 10^{-3}\,[\mathrm{Pa}\cdot\mathrm{s}], \quad 
\rho = 1\,[\mathrm{Kg}/\mathrm{m}^3], \\ 
E = 10^3\,[\mathrm{Pa}], \quad \nu = \frac{1}{3}\,[-], \quad F = 100\,[\mathrm{Pa}],\quad \fb = \bzero, \quad g = 0.
\end{gather*}
The results of the simulations are collected in Figure~\ref{fig:ex04}, where we plot the profiles of pore pressure, horizontal displacement, principal (axial) components of strain and of poroelastic stress over the horizontal mid-line of the domain (at $y = H/2$). The Mandel effect is clearly visible in the first plot, and the difference (between linear and nonlinear cases) in pore pressure build up is similar as the one observed in \cite{li1999nonlinear}, that is, the nonlinear permeability produces a slightly lower pressure. Figure~\ref{fig:ex04b} shows transients of the main variables over time at two spatial points (the left end of the horizontal mid-line and the top-right corner). Qualitatively the results agree with the expected behaviour in both linear and nonlinear regimes. We observe that the largest variation in the first point occurs for the pore pressure drop, whereas for the second point the largest variation is seen in the axial poroelastic stress. For completeness we also depict the deformed configuration at the final time together with the pore pressure distribution (produced with the constant permeability case). 

\begin{figure}[!t]
    \centering
    \includegraphics[width=0.325\textwidth]{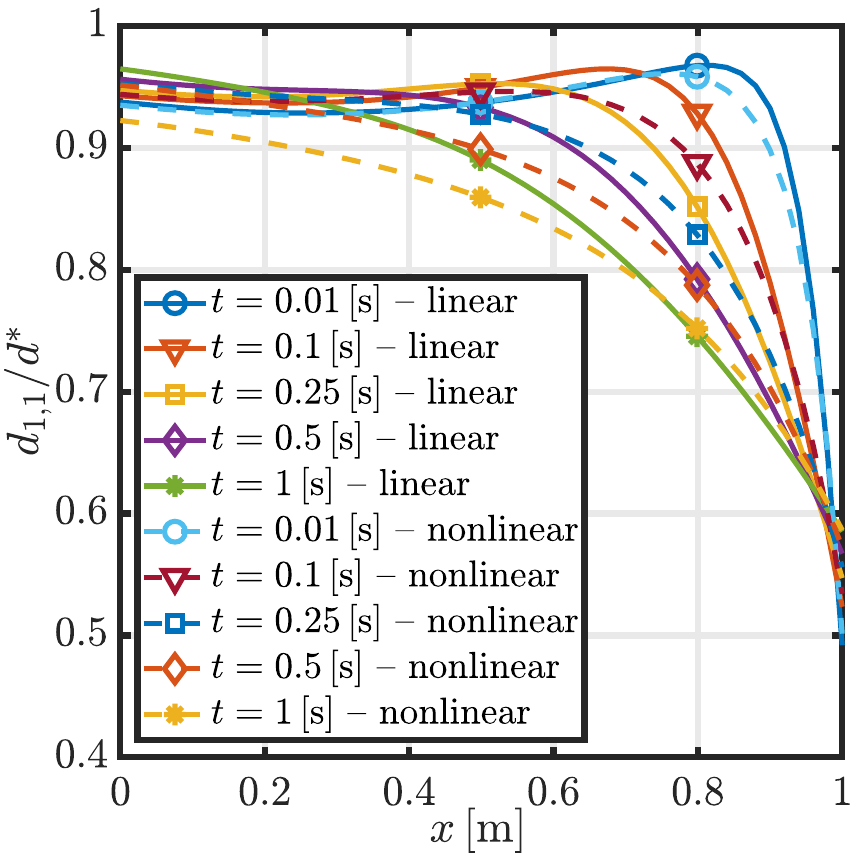}
    \includegraphics[width=0.325\textwidth]{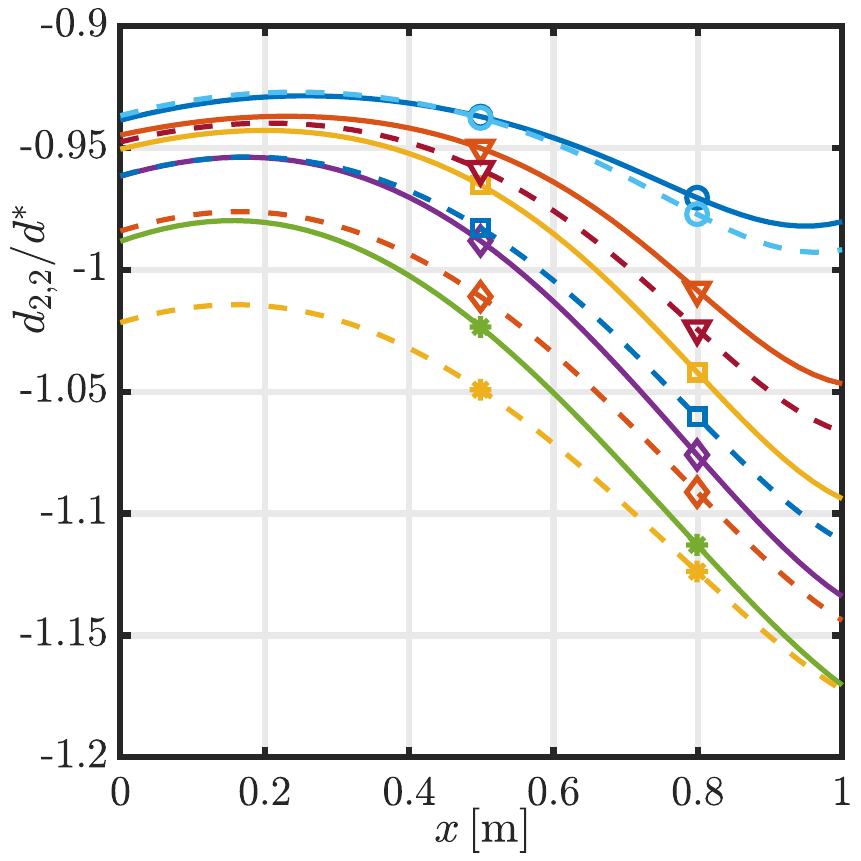}
    \includegraphics[width=0.325\textwidth]{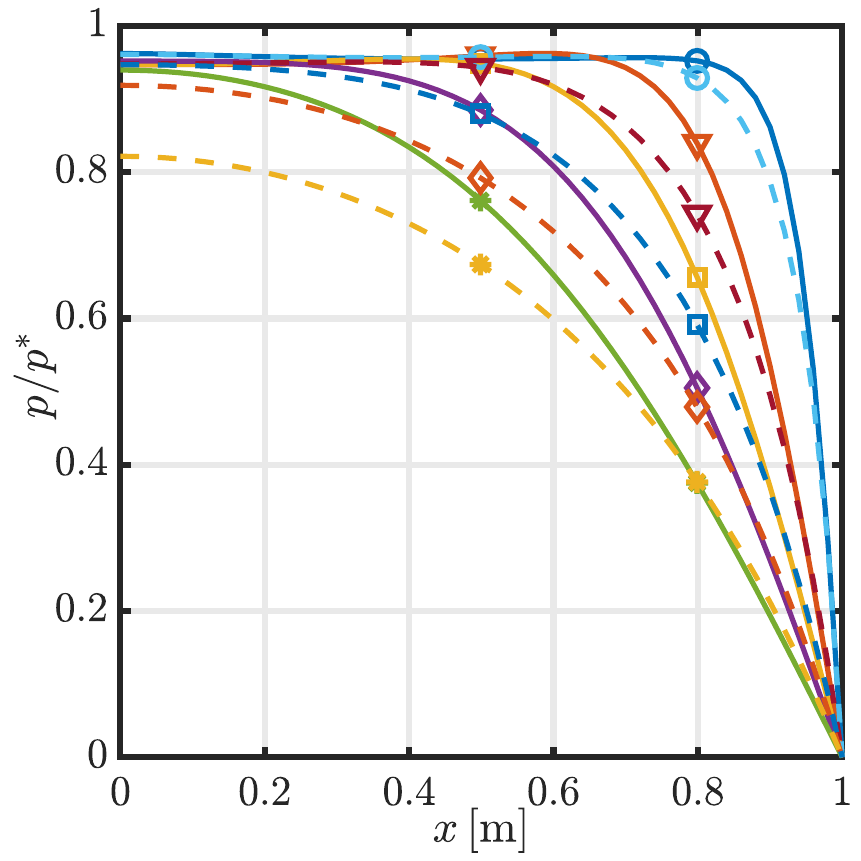}\\
    \includegraphics[width=0.325\textwidth]{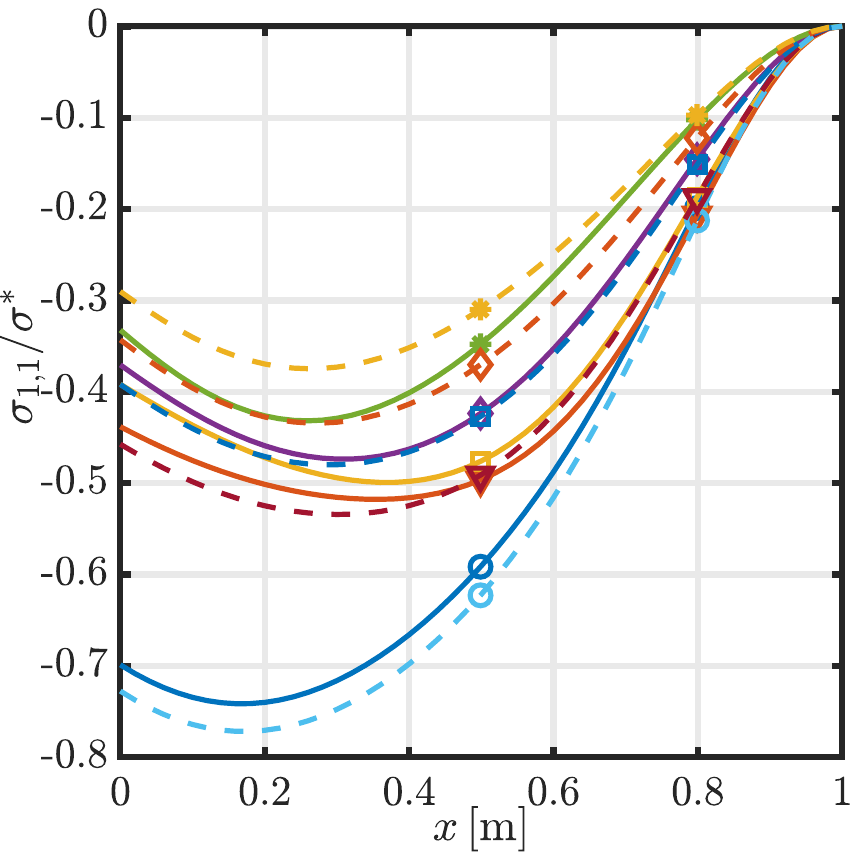}
        \includegraphics[width=0.325\textwidth]{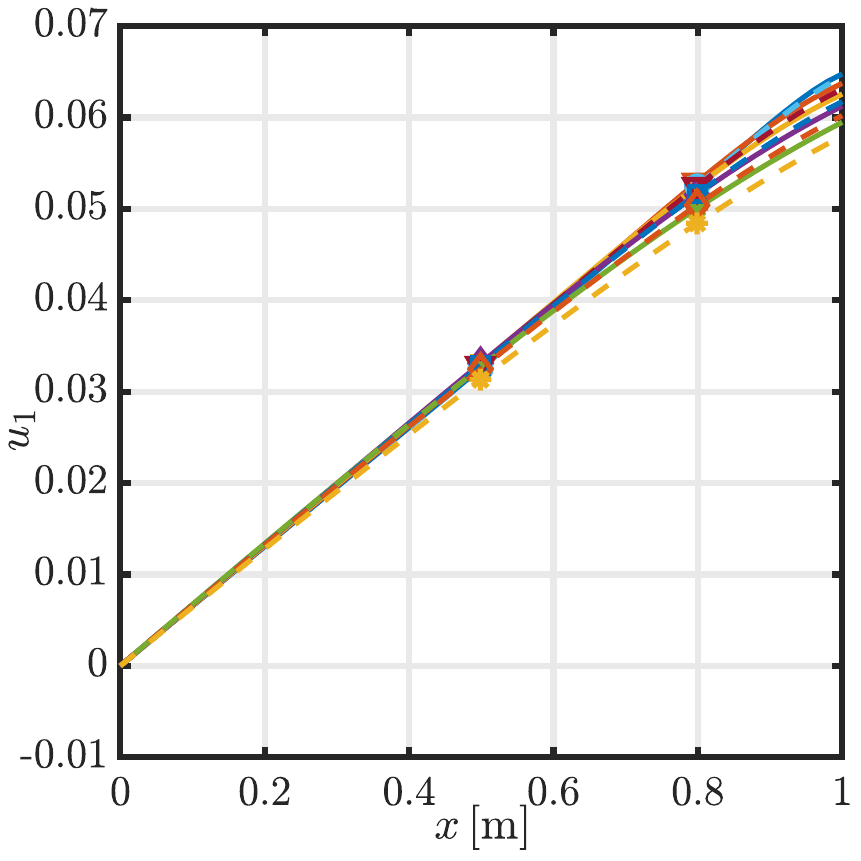}
            \includegraphics[width=0.325\textwidth]{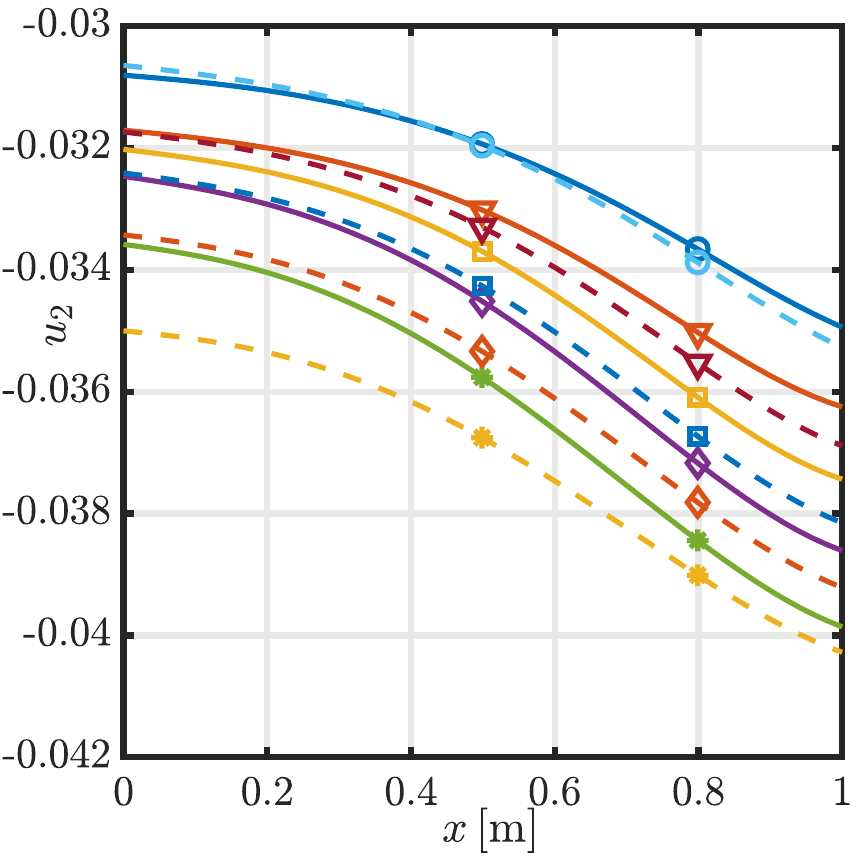}
    \caption{Mandel's test. Plot over the domain horizontal mid-line of axial and radial strains (first and last components of $\bd$) normalised through $d^*=0.07$\,[--], pore pressure profile normalised by $p^*=60$\,[Pa], axial poroelastic stress normalised by $\sigma^* = 4$\,[Pa], and patterns of horizontal an vertical velocities (in [m]), for the constant and nonlinear permeability cases. The legend in the top-left panel (indicating the different times) applies to all panels in the figure.}
    \label{fig:ex04}
\end{figure}

\begin{figure}[!t]
    \centering
    \includegraphics[width=0.325\textwidth]{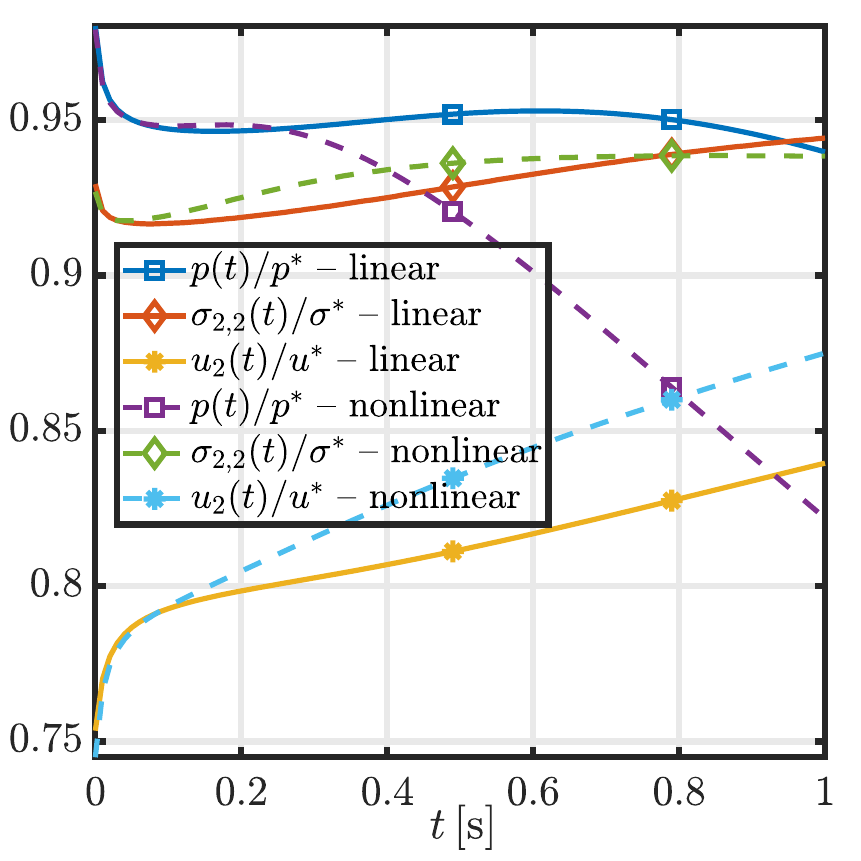}
    \includegraphics[width=0.325\textwidth]{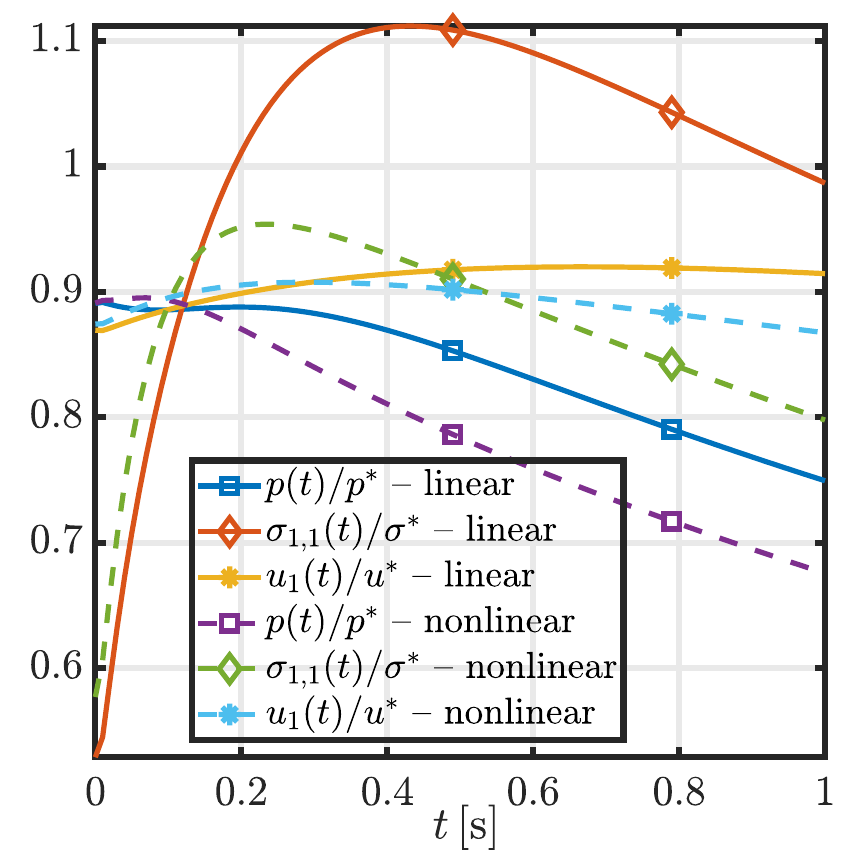}
    {\raisebox{4.5mm}{\includegraphics[width=0.325\textwidth]{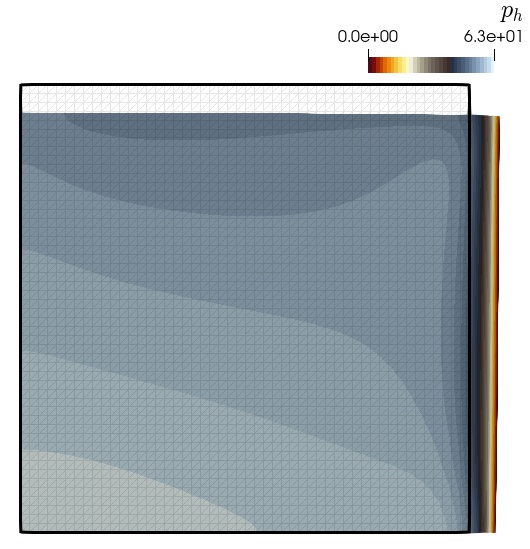}}}
    \caption{Mandel's test. Variation of pore pressure normalised with $p^*=60$\,[Pa], poroelastic stress (axial component normalised with $\sigma^*=-100$\,[Pa] and radial component with $\sigma^*=8$\,[Pa]) and displacement (vertical component normalised with $u^* = -0.04$\,[m] and horizontal component with $u^* = 0.04$\,[m]) against time for the constant and nonlinear permeability cases, and recorded at the points $(0,H/2)$ (left) and $(L/2,H)$ (middle). The right panel shows the patterns of pore pressure at the final time on the deformed configuration.}
    \label{fig:ex04b}
\end{figure}

\bibliographystyle{siam}
\bibliography{total}

\end{document}